\documentclass[11pt,leqno]{article}
\usepackage{graphicx}
\usepackage{amsfonts, amsthm}
\usepackage{amsxtra}
\usepackage{tikz}
\usepackage[latin1]{inputenc}
\usepackage{fontenc}
\usepackage[dvips]{epsfig}
\begin{document}
\newtheorem{theo}{Theorem}[section]
\newtheorem{defi}[theo]{Definition}
\newtheorem{lemm}[theo]{Lemma}
\newtheorem{prop}[theo]{Proposition}
\newtheorem{rem}[theo]{Remark}
\newtheorem{exam}[theo]{Example}
\newtheorem{cor}[theo]{Corollary}
\newtheorem{proper}[theo]{Property}
\newcommand{\mat}[4]{
    \begin{pmatrix}
           #1 & #2 \\
           #3 & #4
      \end{pmatrix}
   }
\def\Z{\mathbb{Z}} 
\def\R{\mathcal{R}} 
\def\I{\mathcal{I}} 
\def\C{\mathbb{C}} 
\def\N{\mathbb{N}} 
\def\PP{\mathbb{P}} 
\def\Q{\mathbb{Q}} 
\def\L{\mathcal{L}} 
\def\ol{\overline} 
\def\bs{\backslash} 
\def\part{P} 

\newcommand{\gcD}{\mathrm {\ gcd}} 
\newcommand{\End}{\mathrm {End}} 
\newcommand{\Aut}{\mathrm {Aut}} 
\newcommand{\GL}{\mathrm {GL}} 
\newcommand{\SL}{\mathrm {SL}} 
\newcommand{\PSL}{\mathrm {PSL}} 
\newcommand{\Mat}{\mathrm {Mat}} 
\newcommand\ga[1]{\overline{\Gamma}_0(#1)} 
\newcommand\pro[1]{\mathbb{P}1(\mathbb{Z}_{#1})} 
\newcommand\Zn[1]{\mathbb{Z}_{#1}}
\newcommand\equi[1]{\stackrel{#1}{\equiv}}
\newcommand\pai[2]{[#1:#2]} 
\newcommand\modulo[2]{[#1]_#2} 
\newcommand\sah[1]{\lceil#1\rceil} 
\def\sol{\phi} 
\begin{center}
{\LARGE\bf Nonrigidity  of piecewise-smooth circle maps}  \\
\vspace{.25in} { \large {Habibulla Akhadkulov\footnote{School of Mathematical Sciences Faculty of Science and Technology University Kebangsaan Malaysia, 43600 UKM Bangi, Selangor Darul Ehsan, Malaysia.  \quad E-mail: akhadkulov@yahoo.com, msn@ukm.my},
 Akhtam Dzhalilov\footnote{Turin Polytechnic University, Kichik Halka yuli 17,  Tashkent 100095, Uzbekistan. \quad E-mail: a\_dzhalilov@yahoo.com},
}Mohd Salmi Md. Noorani$^{2}$}
\end{center}


\title{Nonrigidity  of piecewise smooth circle maps }
\begin{abstract}
Let $f_{i},$  $i=1,2$ be piecewise-smooth  $C^{1}$  circle homeomorphisms with two
break points,  $\log Df_{i},$  $i=1,2$ are absolutely continuous on each continuity intervals
 of $Df_{i}$ and $D\log Df_{i}\in L^{p}$ for some $p>1.$
 Suppose, the jump ratios of $f_{1} $ and   $f_{2} $ at their break points
  do not coincide but have the same total jumps (i.e. the product of jump ratios) and identical irrational rotation
 number of bounded type.  Then the conjugation  $h$ between $f_{1} $ and   $f_{2} $ is a singular function, i.e. it is continuous on $S^1,$ but $Dh(x)=0$ a.e. with respect to Lebesgue measure.
\end{abstract}
\section{Introduction}
This work continues and in some sense completes our study of conjugations between circle homeomorphisms with break type singularities.
Let $S^{1}=\mathbb{R}/\mathbb{Z}$ with clearly defined orientation, metric, Lebesgue measure
and the operation of addition be the \emph{unit circle}. Let $\pi:\mathbb{R}\rightarrow S^{1}$ denote the corresponding
projection mapping that "winds" a straight line on the circle. An arbitrary homeomorphism $f$ that preserves the orientation
of the unit circle $S^{1}$ can "be lifted" on the straight line $\mathbb{R}$ in the form of the homeomorphism
$L_{f}:\mathbb{R}\rightarrow\mathbb{R}$ with property $L_{f}(x+1)=L_{f}(x)+1$ that is connected with $f$ by relation
$\pi\circ L_{f}=f\circ\pi.$ This homeomorphism $L_{f}$ is called the \emph{lift}
of the homeomorphism $f$ and is defined up to an integer term. The most important arithmetic characteristic
 of the homeomorphism $f$ of the unit circle $S^{1}$ is the \emph{rotation number}
 $$
 \rho(f)=\lim_{i\rightarrow\infty}\frac{L^{i}_{f}(x)}{i}\mod1,
 $$
where $L_{f}$ is the lift of $f$ with $S^{1}$ to $\mathbb{R}.$
Here and below, for a given map $F,$ $F^{i}$ denotes its $i$-th iteration.
The classical  Denjoy's theorem  states \cite{De1932}, that
if $f$ is a circle diffeomorphism with irrational rotation number
$\rho=\rho(f)$ and $\log Df$ is of bounded variation, then $f$
is conjugate to the linear rotation $f_{\rho}:x\rightarrow x+\rho\mod 1$, that is, there
exists a unique (up to additional constant) homeomorphism $\varphi$ of the circle
with $f=\varphi^{-1} \circ f_{\rho} \circ \varphi.$
Since the conjugating map
$\varphi$ and the unique $f$-invariant measure $\mu_{f}$ are related by
$\varphi(x)=\mu_{f}([0,x])$ (see \cite{CFS1982}), regularity
properties of the conjugating map $\varphi$ imply
corresponding properties of the density of the absolutely
continuous invariant measure $\mu_{f}$.
 This problem of smoothness of the conjugacy of smooth diffeomorphisms is  now
very well understood (see for instance
\cite{Ar1961,Mo1966,He1979,KO1989.1,KO1989.2,KS1989,Yo1984}).

 A natural extension of  diffeomorphisms of the circle are piecewise-smooth homeomorphisms with break points, that is,
maps that are smooth everywhere except for several singular points at which the first derivative has a jump.
 Notice that Denjoy's result can be extended to circle
homeomorphisms with break points. Below we present the exact statement of the corresponding theorem.
The regularity properties of invariant measures of such maps are quite different from the case
  diffeomorphisms. Namely, invariant measure of  piecewise-smooth circle homeomorphisms with break points and
with irrational rotation number is singular w.r.t. Lebesgue measure (see \cite{DK1998, DL2006, DLM2009, DMS2012}). In this case,
the conjugacy $\varphi$ between $f$ and linear rotation $f_{\rho}$ is
singular function.
Here naturally arises the question on regularity of conjugacy between two circle maps with  break points.
Consider two piecewise-smooth circle  homeomorphisms $f_{1}$, $f_{2}$ which has
break points with the same order on the circle  and
 the same irrational rotation numbers.
On what conditions is the conjugacy  between two such  homeomorphisms  smooth? This is
the rigidity problem for circle homeomorphisms with  break points.
Denote by $\sigma_{f}(b):=Df_{-}(b)/Df_{+}(b)$
the \emph{jump ratio} or \emph{jump} of $f$ at the break point $b.$
The case of circle maps with one break point and  the same jump ratios were studied in detail by K. Khanin and D. Khmelev \cite{KhKm2003},
A. Teplinskii and K. Khanin \cite{TK2004}.
Let $\rho=1/\left(k_{1}+ 1/\left( k_{2}+...+1/\left(k_{n}+...\right) \right)
\right):=[k_1,k_2,\ldots,k_n,\ldots)$ be the
continued fraction expansion of the irrational rotation number
$\rho.$ Define
$$
M_o=\{\rho:\exists C>0, \forall n\in \mathbb{N},\, k_{2n-1}\leq
C\},\,\,\,\,
M_e=\{\rho:\exists C>0, \forall n\in \mathbb{N},\, k_{2n }\leq
C\}.
$$
We formulate the main result of \cite{TK2004}.
 \begin{theo}\label{Tep Kh}
Let  $f_{i}\in C^{2+\alpha}(S^{1}\backslash\{{b_{i}}\}),  \ i=1,2$, $\alpha>0$ be two circle homeomorphisms with one break
point that have the same jump ratio $\sigma$ and the
same irrational rotation number $\rho\in (0, 1).$ In addition, let one of the following
restrictions be true:  either $\sigma> 1$ and $\rho\in M_e$ or $\sigma<1$ and $\rho\in M_o$. Then the map $h$ conjugating the homeomorphisms $f_ 1$ and
$f_2$ is a $C^1$-diffeomorphism.
\end{theo}
In the case homeomorphisms with different jump ratios the following theorem was proved by A. Dzhalilov, H. Akin, S. Temir in \cite{DHT2010}.
\begin{theo}\label{Dz.H.T.}
Let  $f_{i}\in C^{2+\alpha}(S^{1}\backslash\{{b_{i}}\}),  \ i=1,2$, $\alpha>0$  be two circle homeomorphisms with one break
point that have  different jump ratio and the same irrational rotation number $\rho\in (0, 1).$
Then the map $h$ conjugating the homeomorphisms $f_ 1$ and $f_2$ is a singular function.
\end{theo}
Now consider two piecewise-smooth circle homeomorphisms  $f_{1}$ and  $f_{2}$  with $m$ $(m\geq2)$ break
points and the same irrational rotation number.
Denote $BP(f_{1})$ and $ BP(f_{2})$ the sets of  break points of $f_{1}$ and  $f_{2}$.
 \begin{defi}\label{defKO}
The homeomorphisms $f_{1},$ $f_{2}$ are said to be \textbf{break equivalent}
if there exists a bijection  $\psi_{0}$ such that
\begin{enumerate}
  \item [(1)] $\psi_{0}(BP(f_{1}))=BP(f_{2})$;
  \item [(2)] $\sigma_{f_{2}}(\psi_{0}(b))=\sigma_{f_{1}}(b),$ for all $b\in BP(f_{1}).$
\end{enumerate}
\end{defi}
The rigidity problem for the break equivalent $C^{2+\alpha}$-homeomorphisms and
with trivial total jumps (i.e. it is equal to 1) was studied by K. Cunha and D. Smania in \cite{KcDs2012}.
It was proved that any two such homeomorphisms with some combinatorial conditions
are $C^{1}$-conjugated.
The main idea of this work is to consider piecewise-smooth circle homeomorphisms as
 generalized interval exchange transformations.
The case of non break equivalent homeomorphisms with two break points was studied by
H. Akhadkulov, A. Dzhalilov and D. Mayer in \cite{ADM2012}.
 The main result of \cite{ADM2012} is the following theorem.
\begin{theo}\label{ADM }  Let  $f_{i}\in C^{2+\alpha}(S^{1}\backslash\{{a_{i},b_{i}}\}), i=1,2 $
be circle homeomorphisms with two break points $ a_{i},b_{i}.$ Assume that
\begin{enumerate}
  \item [(1)] their rotation numbers $\rho(f_{i}),$ $i=1,2$ are irrational and coincide
  i.e. $\rho(f_{1})=\rho(f_{2})=\rho, \, \rho\in\mathbb{R}^{1}\setminus\mathbb{Q}$;
  \item [(2)] there exists a bijection  $\psi$ such that $\psi(BP(f_{1}))=BP(f_{2})$;
  \item [(3)] $\sigma_{f_{1}}(a_{1})\sigma_{f_{1}}(b_{1})\neq\sigma_{f_{2}}(a_{2})\sigma_{f_{2}}(b_{2})$.
  \end{enumerate}
Then the map $h$ conjugating $f_1$ and $f_2$ is a singular function.
\end{theo}
Now we consider a wider class of circle
homeomorphisms with break points.
We say that a circle homeomorphism  $f$ with finite number break points
satisfies  generalized conditions of Katznelson-Ornstein (K.O),
if  $\log Df$ is absolutely continuous on each continuity intervals
of $Df$ and $D\log Df\in L^{p}$ for some $p>1.$
  In this work we study the conjugating map $h$
between two circle homeomorphisms $f_{1}$ and $f_{2}$ with two
break points and satisfying (K.O) conditions.
Now we formulate the main result of present paper.
\begin{theo}\label{ADM1}
Let $f_{i},$  $i=1,2$ be piecewise-smooth  $C^{1}$  circle homeomorphisms with two
break points $a_{i}, b_{i}$.  Assume that
\begin{itemize}
\item[(1)] the rotation numbers $\rho(f_{i})$ of $f_{i}, \ i=1,2$ are irrational of bounded
 type  and coincide;
\item[(2)]$\sigma_{f_{1}}(a_{1})\sigma_{f_{1}}(b_{1})=\sigma_{f_{2}}(a_{2})\sigma_{f_{2}}(b_{2})$;
  \item[(3)] $\sigma_{f_{1}}(a_{1})\neq\sigma_{f_{2}}(b)$ for all $b\in BP(f_{2});$
\item[(4)] the break points of $f_{i},$  $i=1,2$ do not lie on the same orbit;
\item[(5)] $f_{i},$  $i=1,2$ satisfy (K.O) conditions for the same $p>1$.
\end{itemize}
Then the  map $h$ conjugating $f_{1}$ and $f_{2}$ is a singular function.
\end{theo}

The main approach for proving theorem \ref{ADM1} plays to study the behaviours
of sequence $\Big\{\log \frac{Df_{2}^{q_{n}}(h(x))}{Df^{q_{n}}_{1}(x)}\Big\}_{n}^{\infty}$
where $q_{n},$ $n=1,2,...$ are first return times. This argument has been used
by M. Herman in \cite{He1979} for investigating conjugations between piecewise linear circle
homeomorphisms with two break points.  Recently it has been discussed by
A. Dzhalilov and I. Liousse \cite{DL2006}, to study invariant measures of circle homeomorphisms with two break points.
\section{The Denjoy theory}
 We use the continued fraction
$\rho=[k_{1}, k_{2},...,k_{n},...)$ of the irrational number which is understood as the limit of the sequence of convergents
 $p_{n}/q_{n}=[k_{1}, k_{2},...,k_{n}].$ The sequence of positive integer $k_{n}$ with $n\geq 1,$ which are called
 \emph{incomplete multiples}, is uniquely determined for irrational $\rho.$ The coprimes $p_{n}$ and $q_{n}$ satisfy the recurrence relations
 $p_{n}=k_{n}p_{n-1}+p_{n-2}$ and $q_{n}=k_{n}q_{n-1}+q_{n-2}$ for $n\geq 1,$ where we set for convenience,  $p_{-1}=0,$ $q_{-1}=1,$ and $p_{0}=1,$ $q_{0}=k_{1}.$

 The class of \textbf{P-homeomorphisms}
 consists of orientation preserving circle homeomorphisms
$f$ differentiable except in finite number break points
 at which left and right derivatives, denoted
respectively by $Df_{-}$ and $Df_{+}$, exist, and such that
\begin{itemize}
\item[-] there exist constants $0<c_{1}<c_{2}<\infty$ with
$c_{1}<Df(x)<c_{2}$ for all $x\in S^{1}\backslash{BP(f)}$,
$c_{1}<Df_{-}(x_{b})<c_{2}$ and $c_{1}<Df_{+}(x_{b})<c_{2}$ for
all $x_{b}\in {BP(f)}$, with $BP(f)$ the set
 of break points of $f$ in $S^{1}$;
 \item[-] $\log Df$ has bounded variation in $S^{1}$.
 \end{itemize}
If $\log Df$ has bounded variation, in this situation
 $\log Df^{-1}$ also have the same total variation  and denote by $v=Var_{S^{1}} \log Df.$
Let $\xi\in S^{1},$ we define the $n$-th \emph{generator interval} $\Delta^{n}_{0}(\xi)$ as the circle arc $[\xi, f^{q_{n}}(\xi)]$ for
even $n$ and as  $[f^{q_{n}}(\xi), \xi]$ for odd $n.$
The assertions listed below, which are valid for any P-homeomorphism $f$
 with irrational rotation number $\rho=\rho(f).$
 Their proofs can be found in \cite{ADM2012}, \cite{DL2006} and \cite{He1979}.

\begin{itemize}
   \item [(a)] Generalized Finzi inequality; suppose $\xi\in S^{1},$ $\eta, \zeta \in \Delta^{n-1}_{0}(\xi)$ and $\eta,$ $\zeta$ are continuity points of $Df^{k},$ $0\leq k < q_{n}.$ Then the following inequality holds:
$|\log Df^{k}(\eta)-\log Df^{k}(\zeta)|\leq v.$
   \item [(b)] Generalized Denjoy inequalities; let $\xi_{0}\in S^{1}$ be a continuity point of $Df^{q_{n}},$ then the following inequality holds: $e^{-v}\leq Df^{q_{n}}(\xi_0)\leq e^{v}.$
       \end{itemize}
From generalized Denjoy inequalities it follows that the trajectory of every
point $\xi\in S^1$ is the dense set in $S^{1}$. This together with
monotonicity of the homeomorphism $f$ implies the following theorem.
\begin{itemize}
  \item [(c)]Generalized Denjoy theorem; let  $f$  be a P-homeomorphism with irrational rotation
number $\rho.$  Then $f$ is conjugate to the linear rotation $f_{\rho}$.
\end{itemize}
\begin{rem}\label{remark1}
The same assertions as $(a)-(c)$  holds for $f^{-1}.$
\end{rem}

\section {Absolute continuity of conjugating map}
Consider two P-homeomorphisms on two copies of the circle $S^{1}$ with identical irrational rotation number $\rho$.
Let $\varphi_{1}$ and $\varphi_{2}$ be maps conjugating $f_{1}$ and $f_{2}$ with the pure rotation $f_{\rho}$, i.e. $\varphi_{1}\circ
f_{1}=f_{\rho}\circ\varphi_{1}$ and $\varphi_{2}\circ
f_{2}=f_{\rho}\circ \varphi_{2}$. It is easy to check that the map
$h=\varphi^{-1}_{2}\circ \varphi_{1}$ conjugates $f_{1}$ and
$f_{2}$ , i.e.
\begin{eqnarray}\label{eq21}
h(f_{1}(x))=f_{2}(h(x))
\end{eqnarray}
for all $x\in S^{1}$.

\begin{lemm}\label{Zaruriy}
Let $f_{1}$ and $f_{2}$ are P-homeomorphisms with identical irrational rotation number. Then conjugating map $h$ between  $f_{1}$ and $f_{2}$ is either absolutely continuous or singular function.
\end{lemm}
\begin{proof}
 The conjugating homeomorphism  $h$ is strictly increasing function on $S^{1}.$ Then $Dh$ exists almost everywhere on $S^{1}.$ Denote by $\mathcal{A}=\{x: \,\,x\in S^{1},\,\, Dh(x)>0\}.$ It is clear that the set $\mathcal{A}$ is mod 0 invariant set with respect to  $f_{1}$  i.e. $\mathcal{A}=f^{-1}_{1}(\mathcal{A})$ almost everywhere on $\mathcal{A}$.
 As P-homeomorphism the map $f_{1}$ is  ergodic with respect to Lebesgue measure.
  Hence the Lebesque measure of the set $\mathcal{A}$ is 0 or 1.  The conjugation $h$ is singular function if  $\ell(\mathcal{A})=0$ and  it is absolutely continuous if $\ell(\mathcal{A})=1.$
\end{proof}
The following theorem gives the necessary condition of absolute continuity of conjugation.
\begin{theo}\label{theor 3.1.}
 Let $f_{i}, i=1,2$ are P- homeomorphisms with identical irrational rotation number $\rho$. If conjugation map $h$ between
$f_{1}$ and $f_{2}$ is absolutely continuous function, then for all $\delta>0$
$$
\underset{n \rightarrow \infty}{\lim}\ell (x: |\log Df^{q_n}_{2}(h(x))- \log Df^{q_n}_{1}(x)|\geq \delta)=0.
$$
\end{theo}
\begin{proof}
First we prove that the sequence $f^{q_n}_{1}(x)$ uniformly converges to $x$.
It is clear that  $|f^{q_n}_{1}(x)-x|=|\varphi_{1}^{-1}\circ f_{\rho}^{q_n}\circ\varphi_{1}(x)-x|.$
By setting $y=\varphi_{1}(x)$ we get
$|f^{q_n}_{1}(x)-x|=|\varphi_{1}^{-1}(f_{\rho}^{q_n}(y))-\varphi_{1}^{-1}(y)|.$
Furthermore $|f_{\rho}^{q_n}(x)-x|\leq 1/q_{n}$
does not depend on $x$ and tends to 0. This and the uniform continuity of $\varphi_{1}^{-1}$ on $S^{1}$
implies that the sequence $f^{q_n}_{1}(x)$  uniformly converges to $x$.
Denote by $\Arrowvert\cdot\Arrowvert_{1}$ the norm in
 $L^{1}(S^{1}, d\ell).$
Now we show that
\begin{eqnarray}\label{eq21a}
\underset{n\rightarrow\infty}{\lim} \|\psi\circ f_{1}^{q_n}-\psi\|_{1}=0,\,\,\,
i=1,2\,\,\, \text{for all}\,\,\, \psi \in L^{1}(S^{1}, d\ell).
\end{eqnarray}
Well known fact that the class $C([a, b])$ of continuous functions on $[a, b]$ is dense (in $\|\cdot\|_{1}$) in $L^{1}([a, b], d\ell).$
 From this fact implies that if $\psi\in L^{1}(S^{1},d\ell)$, then for any sufficiently small $\epsilon>0$
there exists a continuous function $\psi_{\epsilon}\in C(S^{1})$ and $\phi_{\epsilon}\in L^{1}(S^{1}, d\ell)$
such that $\psi=\psi_{\epsilon}+\phi_{\epsilon}$ and $\|\phi_{\epsilon}\|_{1}\leq\epsilon.$ Using this and
 Denjoy inequalities we obtain
$$
\|\psi\circ f_{1}^{q_n}-\psi\|_{1}\leq \|\psi_{\epsilon}\circ f_{1}^{q_n}-\psi_{\epsilon}\|_{1}+
(\sup |D f_{1}^{q_n}|^{-1}+1)\|\phi_{\epsilon}\|_{1}\leq
$$
\begin{flushright}
$
\leq\|\psi_{\epsilon}\circ f_{1}^{q_n}-\psi_{\epsilon}\|_{L_{1}}+(1+e^{v})\|\phi_{\epsilon}\|_{1}.
$
\end{flushright}
As $\psi_{\epsilon}$ is uniformly continuous on $S^{1}$ and by exponential refinement $f_{1}^{q_n}(x)$ uniformly tends to $x$,
there exists a positive integer $n_{0}=n_{0}(\epsilon)$ such that for all
$n\geq n_{0},$ the $\|\psi_{\epsilon}\circ f_{1}^{q_n}-\psi_{\epsilon}\|_{1}\leq\epsilon .$
Therefore, $\|\psi\circ f_{1}^{q_n}-\psi\|_{1}\leq(2+e^{v})\epsilon.$
Since $\epsilon>0$ was arbitrary and sufficiently small.

Now we prove theorem \ref{theor 3.1.}.
 Assume that conjugation map $h$ is absolutely continuous function then $D h \in L^{1}(S^{1}, d\ell)$ and $Dh(x)>0,$  $x\in\mathcal{A}.$
  Using equation (\ref{eq21}) it is easy to see that for all natural number $n,$ the function $D h$ satisfies the following
  equation
$$
D h(f_{1}^{q_n}(x))D f_{1}^{q_n}(x)=Df_{2}^{q_n}(h(x)) D h(x)  \,\,\,\, \text{a.e.}
$$
Taking the logarithm, we obtain
$$
\log D h(f_{1}^{q_n}(x))-\log D h(x)=\log Df_{2}^{q_n}(h(x))- \log D f_{1}^{q_n}(x)\,\,\,\,\, \text{a.e.}
$$
multiplying by $2i\pi/2(v_{1}+v_{2}),$ where $v_{j}=Var_{S^{1}}\log Df_{j},$ $j=1,2$ we obtain
$$
\frac{2i\pi\log D h(f_{1}^{q_n}(x))}{2(v_{1}+v_{2})}-\frac{2i\pi\log D h(x)}{2(v_{1}+v_{2})}=
\frac{2i\pi\log Df_{2}^{q_n}(h(x))}{2(v_{1}+v_{2})}-\frac{2i\pi\log D f_{1}^{q_n}(x)}{2(v_{1}+v_{2})}\,\,\,\,\, \text{a.e.}
$$
Consider the following function
$$
\psi(x)=\left\{
  \begin{array}{ll}
    \exp(\frac{2i\pi\log D h(x)}{2(v_{1}+v_{2})}) & \hbox{if \,\,\,\,\,\,\,$x\in \mathcal{A},$}  \\
    0 & \hbox{if \,\,\,\,\,\, $x\in S^{1}\setminus \mathcal{A}$.}
  \end{array}
\right.
$$
It is clear that $\psi$ is measurable, $\psi\in L^{1}(S^{1}, d\ell)$ and $\|\psi\|_{1}=\ell(\mathcal{A}).$
Taking the exponential from last equation we obtain
$$
\psi(f_{1}^{q_{n}}(x))-\psi(x)=\Big[\exp\Big\{2i\pi\Big(\frac{\log Df_{2}^{q_n}(h(x))- \log D f_{1}^{q_n}(x)}{2(v_{1}+v_{2})}\Big)\Big\}-1\Big]\psi(x) \,\,\,\,\,\,\, \text{a.e.}
$$
Integrating the module of this equality, we have
$$
\underset{S^{1}}{\int}|\psi(f_{1}^{q_{n}}(x))-\psi(x)|dx=\underset{S^{1}}{\int}\Big|\exp\Big\{2i\pi\Big(\frac{\log Df_{2}^{q_n}(h(x))- \log D f_{1}^{q_n}(x)}{2(v_{1}+v_{2})}\Big)\Big\}-1\Big|dx=
$$
$$
=\underset{S^{1}}{\int}2\sin\frac{\pi\Big|\log Df_{2}^{q_n}(h(x))- \log D f_{1}^{q_n}(x)\Big|}{2(v_{1}+v_{2})}dx
$$
Suppose, by contradiction, that there exists $\delta>0$ (we may suppose that $\delta\in (0, v_{1}+v_{2}]$) such that
$\ell(S_{\delta}^{n})$ does not converge to 0 when $n$ goes to infinity, where $S_{\delta}^{n}=\{x: |\log Df_{2}^{q_n}(h(x))- \log D f_{1}^{q_n}(x)|\geq\delta\}.$
It follows from Denjoy's inequality  $|\log Df_{2}^{q_n}(h(x))- \log D f_{1}^{q_n}(x)|\leq v_{1}+v_{2},$ therefore
$\pi|\log Df_{2}^{q_n}(h(x))- \log D f_{1}^{q_n}(x)|/2(v_{1}+v_{2})$ and $\pi\delta/2(v_{1}+v_{2})$ belong to $[0, \pi/2],$ an
interval where "sin" is an increasing function. Hence, for all natural $n:$
$$
\underset{S^{1}}{\int} |\psi(f_{1}^{q_n}(x))-\psi(x)|dx \geq \underset{S_{\delta}^{n}}{\int}|\psi(f_{1}^{q_n}(x))-\psi(x)|dx
\geq \Big[2\sin\frac{\pi\delta}{2(v_{1}+v_{2})}\Big] \ell(S_{\delta}^{n}).
$$
 But $\ell(S_{\delta}^{n})$  does not tend to 0 when $n$ goes to $+\infty$. Hence
$\|\psi\circ f_{1}^{q_n}-\psi\|_{1}$ does not tend to 0 when $n$ goes to $+\infty,$
 this is contradicts to (\ref{eq21a}) and  ends the proof of  theorem \ref{theor 3.1.}.
\end{proof}

\section{Dynamical partitions and universal estimates}
In this section we will introduce two types of dynamical partitions and
we will get some estimations for the  ratios of length of elements of these partitions.
 Given a circle homeomorphism $f$ with irrational rotation number $\rho,$ one may
 consider a \emph{positive marked trajectory} (i.e. the positive trajectory of a marked point) $\xi_{i}=f^{i}(\xi_{0})\in S^{1},$ where $i\geq 0,$ and
 pick out of it the sequence of the \emph{dynamical convergents} $\xi_{q_{n}},$ $n\geq 0,$ indexed by the denominators of consecutive rational convergents to $\rho$. We will also conventionally use $\xi_{q_{-1}}=\xi_{0}-1.$ The well-understood arithmetical properties of rational convergents and the combinatorial equivalence between
 $f$ and linear rotation $f_{\rho}$ imply that the dynamical convergents approach the marked point, alternating their order in the following way:
$$
\xi_{q_{-1}}<\xi_{q_{1}}<\xi_{q_{3}}<...<\xi_{q_{2m+1}}<...<\xi_{0}<...<\xi_{q_{2m}}<...<\xi_{q_{2}}<\xi_{q_{0}}.
$$
 For the marked trajectory, we use the notations $\Delta^{n}_{0}=\Delta^{n}_{0}(\xi_{0})$
and $\Delta^{n}_{i}=f^{i}(\Delta^{n}_{0}),$  where $\Delta^{n}_{0}(\xi_{0})$ is $n$-th generator interval.
It is well known, that the set  of
intervals $\textbf{P}_{n}=\textbf{P}_n(\xi_0, f)$ with mutually disjoint interiors defined as
\begin{equation}\label{eq1}
\textbf{P}_{n}=\left\lbrace \Delta_{i}^{n-1}, \ 0\leq i<q_{n};
\ \Delta_{j}^{n}, \ 0\leq j<q_{n-1}\right\rbrace,
\end{equation}
determines a partition of the circle for any $n$.
The partition $\textbf{P}_{n}$ is called the $n$-th \emph{dynamical
partition} of the point
$\xi_{0}.$  Obviously the partition $\textbf{P}_{n+1}$ is a
refinement of the partition $\textbf{P}_{n}$: indeed the intervals of
order $n$ are members of $\textbf{P}_{n+1}$ and each interval $
\Delta_{i}^{n-1}\in \textbf{P}_{n} \, \ 0\leq
i<q_{n},$ is partitioned into $k_{n+1}+1$ intervals belonging to
$\textbf{P}_{n+1}$ such that
\begin{equation}\label{eq2}
\Delta_{i}^{n-1}=\Delta_{i}^{n+1}
\cup\bigcup_{s=0}^{k_{n+1}-1}\Delta_{i+q_{n-1}+sq_{n}}^{n}.
\end{equation}
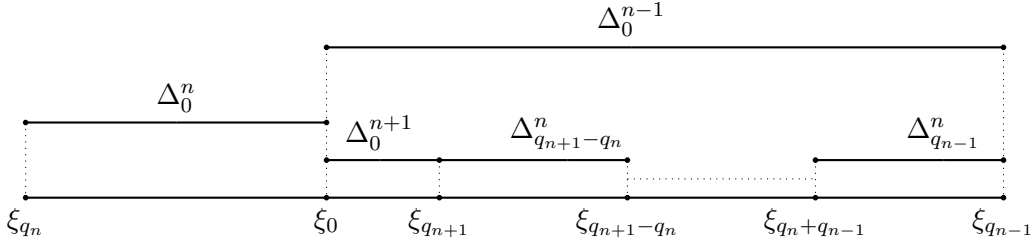
\begin{tikzpicture}
\draw[thick] (0,0) -- (13,0);
\draw[fill=black] (0,0) circle (0.03cm)node[below]{$\xi_{q_{n}}$};
\draw[fill=black] (4,0) circle (0.03cm)node[below]{$\xi_{0}$};
\draw[fill=black] (13,0) circle (0.03cm)node[below]{$\xi_{q_{n-1}}$};
\draw[thick] (0,1) -- (4,1);
\draw[fill=black] (0,1) circle (0.03cm)node[below]{};
\draw[fill=black] (4,1) circle (0.03cm)node[below]{};
\draw[dotted] (0,1)--(0,0);
\draw[thick] (4,2) -- (13,2);
\draw[fill=black] (4,2) circle (0.03cm)node[below]{};
\draw[fill=black] (13,2) circle (0.03cm)node[below]{};
\draw[dotted] (4,2)--(4,0);
\draw[dotted] (13,2)--(13,0);
\draw[fill=white] (2,1.01) circle (0.0cm)node[above]{$\Delta_{0}^{n}$};
\draw[fill=white] (8.05,2.01) circle (0.0cm)node[above]{$\Delta_{0}^{n-1}$};
\draw[thick] (4,0.5) -- (8,0.5);
\draw[fill=black] (4,0.5) circle (0.03cm)node[below]{};
\draw[fill=black] (8,0.5) circle (0.03cm)node[below]{};
\draw[dotted] (5.5,0.5)--(5.5,0);
\draw[fill=black] (5.5,0.5) circle (0.03cm)node[below]{};
\draw[fill=black] (5.5,0) circle (0.03cm)node[below]{$\xi_{q_{n+1}}$};
\draw[fill=white] (4.7,0.501) circle (0.0cm)node[above]{$\Delta_{0}^{n+1}$};
\draw[dotted] (8,0.5)--(8,0);
\draw[fill=white] (7.2,0.501) circle (0.0cm)node[above]{$\Delta^{n}_{q_{n+1}-q_{n}}$};
\draw[dotted] (10.5,0.5)--(10.5,0);
\draw[thick] (10.5,0.5) -- (13,0.5);
\draw[fill=black] (10.5,0.5) circle (0.03cm)node[below]{};
\draw[fill=black] (13,0.5) circle (0.03cm)node[below]{};
\draw[fill=black] (8,0) circle (0.03cm)node[below]{$\xi_{q_{n+1}-q_{n}}$};
\draw[fill=black] (10.5,0) circle (0.03cm)node[below]{$\xi_{q_{n}+q_{n-1}}$};
\draw[fill=white] (12.2,0.501) circle (0.0cm)node[above]{$\Delta^{n}_{q_{n-1}}$};
\draw[dotted] (8,0.25) -- (10.5,0.25);
\draw[fill=white] (6.2, -1) circle (0.0cm)node[below]{\emph{The transition scheme from dynamical partition $\textbf{P}_{n}$
to $\textbf{P}_{n+1}$ for the case $i=0.$}};
\end{tikzpicture}
\\
\begin{defi}\label{Def1}
 Let  $0<K\leq 1$. We call two intervals of $S^{1}$ are \textbf{K-comparable} if the ratio of their lengths in $[K, K^{-1}].$
\end{defi}
 Let $f$  be a P-homeomorphism with irrational rotation number $\rho=\left[k_{1},k_{2},...,k_{n},...\right)$ of bounded type.
Here and later we denote by $\mathcal{Q}=\mathcal{Q}(\rho)=\underset{n}{\sup} (k_{n}).$
\begin{proper}\label{proper1}
 There exists universal constant $C_{2}=C_{2}(\mathcal{Q}, f)$ such that two consecutive atoms of  $\textbf{P}_{n}$
 are $C_{2}$-comparable.
 \end{proper}
\begin{proof}
Two consecutive atoms of  $\textbf{P}_{n}$ can be of the following three types:
$$
\text{(I)}\,\ \Delta_{i}^{n-1}\ \text{and}\, \,\ \Delta_{i+q_{n-1}}^{n-1}, \,\,\  \text{(II)}\,\
\Delta_{i}^{n}\,\ \text{and} \,\ \Delta_{i}^{n-1}, \,\,\ \text{(III)}\,\
\Delta_{i}^{n}\,\ \text{and}\,\ \Delta_{i+q_{n}-q_{n-1}}^{n-1}.
$$
In (I), the consecutive atoms are $I$ and $f^{q_{n-1}}(I).$ Since, $f$ has finite number break points, by the mean value theorem, we have
$$
\inf_{I}(Df^{q_{n-1}})\leq \frac{|f^{q_{n-1}}(I)|}{|I|}\leq \sup_{I}(Df^{q_{n-1}}).
$$
From Denjoy  inequalities, it follows that $e^{-v}\leq |f^{q_{n-1}}(I)|/|I|\leq e^{v}.$
Consider the case (II). Using equality (\ref{eq2}) we get
$$
\frac{|\Delta_{i}^{n-1}|}{|\Delta_{i}^{n}|}=
\frac{|\Delta_{i}^{n+1}|
+\underset{s=0}{\overset{k_{n+1}-1}{\sum}}|\Delta_{i+q_{n-1}+sq_{n}}^{n}|}
{|\Delta_{i}^{n}|}.
$$
It is clear that $\Delta_{i}^{n+1}\subset \Delta_{-q_{n}+i}^{n}.$ Using similar arguments (I) we have
$$
\sum_{s=0}^{k_{n+1}-1}\underset{\Delta_{i}^{n}}{\inf}(Df^{q_{n-1}+sq_{n}})\leq\frac{|\Delta_{i}^{n-1}|}{|\Delta_{i}^{n}|}\leq \sum_{s=0}^{k_{n+1}-1}\underset{\Delta_{i}^{n}}{\sup}(Df^{q_{n-1}+sq_{n}})+\underset{\Delta_{i}^{n}}{\sup}(Df^{-q_{n}}).
$$
Using Denjoy inequalities, we get $\inf(Df^{q_{n-1}+sq_{n}})\geq e^{-(s+1)v}$ and
$\sup(Df^{q_{n-1}+sq_{n}})\leq e^{(s+1)v}.$ Furthermore, the rotation number is bounded type i.e. $k_{n+1}\leq \mathcal{Q},$ hence it follows $|\Delta_{i}^{n-1}|$ and $|\Delta_{i}^{n}|$ are $[(\mathcal{Q}+1)e^{(\mathcal{Q}+1)v}]^{-1}$-comparable.
Now, we consider the case (III). By Denjoy inequalities intervals
$|\Delta_{i}^{n-1}|$ and $|\Delta_{i+q_{n}-q_{n-1}}^{n-1}|$
are $e^{-2v}$ -comparable. Using above argument $|\Delta_{i}^{n-1}|$ and $|\Delta_{i}^{n}|$ are $[(\mathcal{Q}+1)e^{(\mathcal{Q}+1)v}]^{-1}$-comparable.  Finally, any two consecutive atoms $\textbf{P}_{n}$ are $C_{2}=[(\mathcal{Q}+1)e^{(\mathcal{Q}+3)v}]^{-1}$-comparable.
\end{proof}
\begin{proper}\label{proper2}  There exists universal constant $C_{3}=C_{3}(\mathcal{Q}, f)\leq 1$ such that an atom $\Delta^{n+1}$ of $\textbf{P}_{n+1}$
 is $C_{3}$-comparable to the atom $\Delta^{n}$ of $\textbf{P}_{n}$ that contains it.
 \end{proper}
\begin{proof}
The atom $\Delta^{n+1}$ of $\textbf{P}_{n+1}$  can be of following three types:
\begin{enumerate}
  \item[-] either $\Delta_{i}^{n}$ \,such that\,\ $\Delta_{i}^{n}\in \textbf{P}_{n}$ or
  \item[-] $\Delta_{i+q_{n-1}+sq_{n}}^{n}$\,\ such that\,\
  $\Delta_{i+q_{n-1}+sq_{n}}^{n}\subset \Delta_{i}^{n-1}$ \, and \,$\Delta_{i}^{n-1}\in \textbf{P}_{n}$ or
  \item[-]  $\Delta^{n+1}_{i}$ \,\ such that \,\ $\Delta^{n+1}_{i}\subset \Delta^{n-1}_{i}$ \,\ and \,\
      $\Delta^{n-1}_{i} \in \textbf{P}_{n}.$
\end{enumerate}
In the first case the atoms $\Delta^{n+1}$ and $\Delta^{n}$ are 1-comparable. In the second  case, using similar argument to the above property
we get the atoms $\Delta^{n+1}$ and $\Delta^{n}$ are $[(\mathcal{Q}+1)e^{(\mathcal{Q}+1)v}]^{-1}$-comparable
and the third case we get the atoms $\Delta^{n+1}$ and $\Delta^{n}$ are $[(\mathcal{Q}+1)e^{(\mathcal{Q}+1)v}]^{-2}$-comparable.
Finally, if we take $C_{3}=[(\mathcal{Q}+1)e^{(\mathcal{Q}+1)v}]^{-2}$ then we are done.
\end{proof}
Note that in the second and third cases of the above property the number $C_{3}$
is the greatest lower bound but $C^{-1}_{3}$ can not be the least upper bound. Using
Denjoy inequalities and relation (\ref{eq2})  can be found the number $C_{4}=(1+e^{-v})^{-1}$ is the
least upper bound.
\begin{proper}\label{inter nisbat va expon}
There exist constants $0<\kappa<\lambda<1$  such that
 for all $n,\, m \in\mathbb{N}$ holds the following inequality
  $\kappa^{m}|\Delta^{n}_{0}|\leq|\Delta^{n+m}_{0}|\leq (1+e^{v}) \lambda^{m} |\Delta^{n}_{0}|$.
\end{proper}
\begin{proof}
 By definition of dynamical partition, it is easy to see that  $|\Delta_{0}^{n-1}|\geq |\Delta_{0}^{n+1}|+|\Delta_{q_{n+1}-q_{n}}^{n}|$ and $\Delta_{0}^{n+1}\subset \Delta_{q_{n+1}}^{n}.$  Using property \ref{proper2} and proper last two relations together with
  Denjoy inequalities implies
$$
C_{3}^{-1}\geq\frac{|\Delta_{0}^{n-1}|}{|\Delta_{0}^{n+1}|}\geq 1+\frac{|\Delta_{q_{n+1}-q_{n}}^{n}|}{|\Delta_{q_{n+1}}^{n}|}\geq 1+e^{-v}.
$$
By induction way we can show that if $m$ is even then
$$
C_{3}^{\frac{m}{2}}|\Delta_{0}^{n}|\leq |\Delta_{0}^{n+m}|\leq  (1+e^{-v})^{-\frac{m}{2}}|\Delta_{0}^{n}|.
$$
 If $m$ is odd then, using property \ref{proper1} we have
$$
C_{3}^{\frac{m-1}{2}}C_{2}|\Delta_{0}^{n}|\leq |\Delta_{0}^{n+m}|\leq  e^{v}(1+e^{-v})^{-\frac{m-1}{2}}|\Delta_{0}^{n}|.
$$
If we take $\kappa=\min\{\sqrt{C_{3}}, C_{2}\}$ and  $\lambda=1/\sqrt{1+e^{-v}}$
 then we are done.
\end{proof}
Apply at most three times Finzi inequality to property \ref{inter nisbat va expon}  we get following remark.
\begin{rem}\label{inter nisbat va expon rem}
 For all $n,\, m \in\mathbb{N}$ holds the following inequality
  $e^{-3v}\kappa^{m}|\Delta^{n}|\leq|\Delta^{n+m}|\leq (1+e^{v})e^{3v} \lambda^{m} |\Delta^{n}|,$
   where $\Delta^{n}\in \textbf{P}_{n},$
$\Delta^{n+m}\in \textbf{P}_{n+m}$ and $\Delta^{n+m}\subset\Delta^{n}.$
\end{rem}
Using this remark we show that the oscillation of $\log Df^{k}$
tends to zero with exponential fast on every exponential small continuity intervals of $Df^{k}$.

\begin{lemm}\label{universal estimates}\textbf{(Universal estimates)}  Let $\log Df$ be an absolutely
continuous each  continuity interval of $Df$ and  $D \log Df \in L_{p}$ for some $p>1,$
then for all $n,$ $l$ and natural numbers, for all integer $k$ such that $0\leq k\leq q_{n}$ and
 for all $\xi\in S^{1},$ $\eta \in \Delta_{0}^{n+l}(\xi)$ in the same continuity interval of $Df^{k},$ there exists a universal
 constant $C_{5}=C_{5}(f)>0$ such that
$$
|\log Df^{k}(\xi)-\log Df^{k}(\eta)|\leq C_{5}\lambda^{l/q},
$$
 where $q^{-1}=1-p^{-1}.$
\end{lemm}
\begin{proof}
Fix $n\in \mathbb{N}$ and $k\in \{0, 1, ..., q_{n}\}.$ Let
$\xi\in S^{1}$ and  $\eta \in \Delta_{0}^{n+l}(\xi)$  be two circle points
lying in the same continuity interval of $Df^{k}.$ Then, we have
$$
|\log Df^{k}(\xi)-\log Df^{k}(\eta)|\leq \sum_{j=0}^{k-1}|\log Df(f^{k}(\xi))-\log Df(f^{k}(\eta))|
\leq
$$
$$
\leq\sum_{j=0}^{k-1}\int_{f^{k}(\eta)}^{f^{k}(\xi)}|D\log Df(s)|ds=\int|1_{U}D\log Df(s)|ds\leq \|1_{U}\|_{q}\|D\log Df\|_{p},
$$
where $U=\underset{j=0}{\overset{k-1}{\bigcup}}f^{j}([\xi, \eta]).$
Using remark \ref{inter nisbat va expon rem} we get
$$
\ell(U)=\sum_{j=0}^{k-1}|f^{k}(\eta)-f^{k}(\xi)|\leq \sum_{j=0}^{k-1}\frac{|\Delta^{n+l}_{j}(\xi)|}{|\Delta^{n}_{j}(\xi)|}
|\Delta^{n}_{j}(\xi)|\leq (1+e^{v})e^{3v} \lambda^{l}.
$$
If we take $C_{5}=[(1+e^{v})e^{3v}]^{\frac{1}{q}} \|D\log Df\|_{p}$ then we are done.
\end{proof}
Now we introduce a new  partition $\textbf{D}_{n}=\textbf{D}_{n}(\xi_0, f)$ of circle.
It will be known at section 6, that why we need introduce the  new  partition $\textbf{D}_{n}.$
So, we consider a \emph{full marked trajectory} (i.e. the full trajectory of a marked point) $\xi_{i}=f^{i}(\xi_{0})\in S^{1},$ $i\in \mathbb{Z}$ and
 pick out of it the sequence of the dynamical convergents $\xi_{\pm q_{n}},$ $n\geq 0.$
 It is well known
$$
\xi_{q_{-1}}<...<\xi_{q_{2m-1}}<\xi_{-q_{2m}}<\xi_{q_{2m+1}}<...<\xi_{0}<...<\xi_{q_{2m}}<\xi_{-q_{2m-1}}<\xi_{q_{2m-2}}<...<\xi_{q_{0}}.
$$
Using the $n$-th fundamental interval of partition $\textbf{P}_{n}=\textbf{P}_{n}(\xi_0, f),$ we define the
$n$-th fundamental intervals of dynamical partition $\textbf{D}_{n}$ as the following:
$I^{n}_{0}=\Delta^{n}_{0}(\xi_{-q_{n}})$ and $I^{n-1,n}_{0}=\Delta^{n-1}_{0}\setminus \Delta^{n}_{0}(\xi_{-q_{n}}).$
It is well known, that the set of
intervals
\begin{equation}\label{eq3}
\textbf{D}_{n}=\left\lbrace I_{i}^{n-1,n}, \ 0\leq i<q_{n};
\ I_{j}^{n}, \ 0\leq j<q_{n}+q_{n-1}\right\rbrace
\end{equation}
with mutually disjoint interiors defined as
determines a partition of the circle for any $n$.
It is clear that the partition $\textbf{D}_{n}$ is a subpartition of $\textbf{P}_{n}$
obtained by adding some negative iterates of $\xi_{0}$.
  The partition $\textbf{D}_{n+1}$ is a refinement of the partition
 $\textbf{D}_{n}.$ As we go to $\textbf{D}_{n+1}$ the following occurs.
Each of intervals $I_{j}^{n}, \ 0\leq j<q_{n}+q_{n-1}$
is partitioned into two intervals belonging to $\textbf{D}_{n+1},$ one of them $I^{n+1}$
form and second is $I^{n,n+1}$ form. Similarly, each of intervals $I_{i}^{n-1,n}, \ 0\leq i<q_{n}$
is partitioned into $2k_{n+1}-1$ intervals belonging to $\textbf{D}_{n+1},$
$k_{n+1}-1$ of them  $I^{n,n+1}$ form and $k_{n+1}$
of them $I^{n+1}$ form.
Moreover, any two consecutive subintervals of $I_{i}^{n-1,n}, \ 0\leq i<q_{n}$ are  $I^{n,n+1}$  and
 $I^{n+1}$ form. Each $I^{n+1}$ (or $I^{n,n+1}$) form interval goes to the next $I^{n+1}$ (or $I^{n,n+1}$)
 form interval by $f^{\pm q_{n}}.$
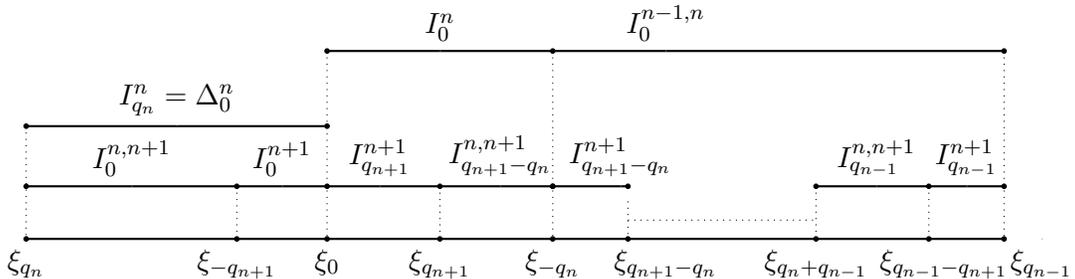
\begin{tikzpicture}
\draw[thick] (0,0) -- (13,0);
\draw[fill=black] (0,0) circle (0.03cm)node[below]{$\xi_{q_{n}}$};
\draw[fill=black] (4,0) circle (0.03cm)node[below]{$\xi_{0}$};
\draw[fill=black] (13,0) circle (0.03cm)node[below]{};
\draw[fill=black] (13.5,0) circle (0.0cm)node[below]{$\xi_{q_{n-1}}$};
\draw[thick] (0,1.5) -- (4,1.5);
\draw[fill=black] (0,1.5) circle (0.03cm)node[below]{};
\draw[fill=black] (4,1.5) circle (0.03cm)node[below]{};
\draw[dotted] (0,1.5)--(0,0);
\draw[thick] (4,2.5) -- (13,2.5);
\draw[fill=black] (4,2.5) circle (0.03cm)node[below]{};
\draw[fill=black] (13,2.5) circle (0.03cm)node[below]{};
\draw[dotted] (4,2.5)--(4,0);
\draw[dotted] (13,2.5)--(13,0);
\draw[fill=white] (2,1.51) circle (0.0cm)node[above]{$I_{q_{n}}^{n}=\Delta_{0}^{n}$};

\draw[thick] (0,0.7) -- (4,0.7);
\draw[dotted] (2.8,0.7)--(2.8,0);
\draw[fill=black] (0, 0.7) circle (0.03cm)node[below]{};
\draw[fill=black] (2.8, 0.7) circle (0.03cm)node[below]{};
\draw[fill=black] (2.8,0) circle (0.03cm)node[below]{$\xi_{-q_{n+1}}$};
\draw[fill=white] (3.4,0.701) circle (0.0cm)node[above]{$I^{n+1}_{0}$};
\draw[fill=white] (1.4,0.701) circle (0.0cm)node[above]{$I^{n,n+1}_{0}$};

\draw[fill=white] (5.5,2.501) circle (0.0cm)node[above]{$I_{0}^{n}$};
\draw[fill=white] (8.5,2.501) circle (0.0cm)node[above]{$I_{0}^{n-1,n}$};
\draw[fill=black] (7,2.5) circle (0.03cm)node[below]{};
\draw[thick] (4,0.7) -- (8,0.7);
\draw[fill=black] (4,0.7) circle (0.03cm)node[below]{};
\draw[fill=black] (7,0.7) circle (0.03cm)node[below]{};
\draw[fill=white] (6.3,0.701) circle (0.0cm)node[above]{$I^{n,n+1}_{q_{n+1}-q_{n}}$};
\draw[fill=white] (7.9,0.701) circle (0.0cm)node[above]{$I^{n+1}_{q_{n+1}-q_{n}}$};
\draw[dotted] (5.5,0.7)--(5.5,0);
\draw[fill=black] (5.5,0.7) circle (0.03cm)node[below]{};
\draw[fill=black] (5.5,0) circle (0.03cm)node[below]{$\xi_{q_{n+1}}$};
\draw[fill=white] (4.7,0.701) circle (0.0cm)node[above]{$I_{q_{n+1}}^{n+1}$};
\draw[dotted] (8,0.5)--(8,0);
\draw[dotted] (10.5,0.7)--(10.5,0);
\draw[dotted] (12,0.7)--(12,0);
\draw[fill=black] (12,0.7) circle (0.03cm)node[below]{};
\draw[fill=black] (12,0) circle (0.03cm)node[below]{};
\draw[fill=black] (12.2,0) circle (0.0cm)node[below]{$\xi_{q_{n-1}-q_{n+1}}$};
\draw[thick] (10.5,0.7) -- (13,0.7);
\draw[fill=black] (10.5,0.7) circle (0.03cm)node[below]{};
\draw[fill=black] (13,0.7) circle (0.03cm)node[below]{};
\draw[fill=black] (8,0) circle (0.03cm)node[below]{};
\draw[fill=black] (8.5,0) circle (0.0cm)node[below]{$\xi_{q_{n+1}-q_{n}}$};
\draw[fill=black] (8,0.7) circle (0.03cm)node[below]{};

\draw[fill=black] (7,0) circle (0.03cm)node[below]{$\xi_{-q_{n}}$};
\draw[dotted] (7,2.5) -- (7,0);
\draw[fill=black] (7,0.7) circle (0.03cm)node[below]{};
\draw[fill=black] (10.5,0) circle (0.03cm)node[below]{$\xi_{q_{n}+q_{n-1}}$};
\draw[fill=white] (12.5,0.701) circle (0.0cm)node[above]{$I^{n+1}_{q_{n-1}}$};
\draw[fill=white] (11.3,0.701) circle (0.0cm)node[above]{$I^{n,n+1}_{q_{n-1}}$};
\draw[dotted] (8,0.25) -- (10.5,0.25);
\draw[fill=white] (6.2, -1) circle (0.0cm)node[below]{\emph{The transition scheme from dynamical partition $\textbf{D}_{n}$
to $\textbf{D}_{n+1}$ for the case $i=0$}.};
\end{tikzpicture}
\\
The partition  $\textbf{D}_{n}$ is a subpartition of $\textbf{P}_{n}$
in the sense that each interval of the partition
 $\textbf{P}_{n}$ consists of entire of intervals of the partition $\textbf{D}_{n}.$
 Now we prove some properties of the partition $\textbf{D}_{n}$.
\begin{proper}\label{proper3}
There exists universal constant $C_{6}=C_{6}(\mathcal{Q}, f)\leq 1$ such that
any interval of $\textbf{D}_{n}$ is $C_{6}$-comparable to the interval of $\textbf{P}_{n}$ that contains it.
\end{proper}
 \begin{proof}
By definitions of dynamical partitions $\textbf{P}_{n}$ and $\textbf{D}_{n},$ it is easy to see that between
intervals of  dynamical partitions $\textbf{P}_{n}$ and $\textbf{D}_{n}$ has following relations:
\begin{equation}\label{eq4}
\Delta_{j}^{n-1}=I_{j}^{n}\cup I_{j}^{n-1,n}, \,\, 0\leq j<q_{n}
\,\,\, \text{and}\,\,\, \Delta_{i}^{n}=I_{i+q_{n}}^{n}, \ 0\leq i<q_{n-1}.
\end{equation}
From this relations implies that the intervals $\Delta_{i}^{n}$ and $I_{i+q_{n}}^{n}, \ 0\leq i<q_{n-1}$
are $1$-comparable. Now we prove that both intervals  $I_{j}^{n}$ and $I_{j}^{n-1,n}$ comparable to the
$\Delta_{j}^{n-1},\,\, 0\leq j<q_{n}.$  It is clear that
$$
\underset{s=0}{\overset{k_{n+1}-1}{\bigcup}}f^{-sq_{n}}(I_{j}^{n})
\subset\Delta_{j}^{n-1}\subset\underset{s=0}{\overset{k_{n+1}}{\bigcup}}f^{-sq_{n}}(I_{j}^{n}).
$$
Considering this and using Denjoy inequalities together with relations (\ref{eq4}) we get
\begin{equation}\label{eq5}
1+e^{-v}\leq
\frac{|\Delta_{j}^{n-1}|}{|I_{j}^{n}|}\leq 1+\mathcal{Q}e^{\mathcal{Q}v}\,\,\,
\text{and}\,\,\
1+\frac{1}{\mathcal{Q}e^{\mathcal{Q}v}}\leq
\frac{|\Delta_{j}^{n-1}|}{|I_{j}^{n-1,n}|}\leq 1+e^{v}.
\end{equation}
Finally, if we take $C_{6} = [1+\mathcal{Q}e^{\mathcal{Q}v}]^{-1}$ then we are done.
\end{proof}
Since, the partition  $\textbf{D}_{n}$ is a subpartition of $\textbf{P}_{n}$
using property \ref{proper1} and property \ref{proper3} we get the following remark.
\begin{rem}\label{con inter}
Any two consecutive intervals of  $\textbf{D}_{n}$
 are $C_{2}C_{6}^{2}$-comparable.
 \end{rem}
\begin{proper}\label{proper4}  There exists universal constants $C_{7}=C_{7}(\mathcal{Q}, f)$ and $C_{8}=C_{8}(\mathcal{Q}, f)$ which is
 $0<C_{7}<C_{8}\leq 1$ such that
 $
 C_{7}|I^{n}|\leq |I^{n+1}|\leq C_{8}|I^{n}|
 $
 where $I^{n+1}\in\textbf{D}_{n+1},$
$I^{n}\in\textbf{D}_{n}$ and $I^{n+1}\subset I^{n}.$
\end{proper}
\begin{proof}
First we obtain all this $|I^{n+1}_{i}|/|I^{n,n+1}_{i}|,$ $0\leq i<q_{n+1}$
ratios, where  $I^{n+1}_{i}$ and $I^{n,n+1}_{i}$ are two consecutive intervals of  $\textbf{D}_{n+1}.$
It is clear that
\begin{equation}\label{eq6}
\underset{s=1}{\overset{k_{n+2}-1}{\bigcup}}f^{sq_{n+1}}(I_{i}^{n+1})
\subset I_{i}^{n,n+1}\subset\underset{s=1}{\overset{k_{n+2}}{\bigcup}}f^{sq_{n+1}}(I_{i}^{n+1}).
\end{equation}
Using Denjoy inequalities we get
\begin{equation}\label{eq7}
(\mathcal{Q}e^{\mathcal{Q}v})^{-1}|I^{n,n+1}_{i}|\leq |I^{n+1}_{i}|\leq e^{v}|I^{n,n+1}_{i}|.
\end{equation}
If  $I^{n}=I^{n}_{i}\in\textbf{D}_{n},$ $0\leq i<q_{n}+q_{n-1}$ then
the interval $I^{n}_{i}$ is partitioned into two intervals belonging to $\textbf{D}_{n+1},$
one of them $I^{n+1}$ form and second is $I^{n,n+1}$ form. Using (\ref{eq7}) we have
 \begin{equation}\label{eq8}
\frac{1}{1+\mathcal{Q}e^{\mathcal{Q}v}}\leq \frac{|I^{n+1}|}{|I^{n}_{i}|}\leq \frac{e^{v}}{1+e^{v}}\,\,\,\text{and}\,\,\,
\frac{1}{1+e^{v}}\leq \frac{|I^{n,n+1}|}{|I^{n}_{i}|}\leq \frac{\mathcal{Q}e^{\mathcal{Q}v}}{1+\mathcal{Q}e^{\mathcal{Q}v}}.
\end{equation}
If  $I^{n}=I_{i}^{n-1,n}\in\textbf{D}_{n},$ $0\leq i<q_{n}$ then
the interval $I_{i}^{n-1,n}$ is partitioned into $2k_{n+1}-1$ intervals belonging to $\textbf{D}_{n+1},$
$k_{n+1}-1$ of them  $I^{n,n+1}$ form and $k_{n+1}$ of them $I^{n+1}$ form.
In particularly if $k_{n+1}=1$ then $I^{n+1}=I_{i}^{n-1,n},$ if $k_{n+1}\geq 2$ then
for any two consecutive subintervals $I^{n+1},$ $I^{n,n+1}$ of $I_{i}^{n-1,n}$
using properties of  partition $\textbf{D}_{n+1}$ and Denjoy inequalities
we have
 \begin{equation}\label{eq9}
 |I^{n+1}|(\underset{s=0}{\overset{k_{n+1}-1}{\sum}}e^{-sv})+
|I^{n,n+1}|(\underset{s=0}{\overset{k_{n+1}-2}{\sum}}e^{-sv})
\leq I_{i}^{n-1,n}\leq |I^{n+1}|(\underset{s=0}{\overset{k_{n+1}-1}{\sum}}e^{sv})+
|I^{n,n+1}|(\underset{s=0}{\overset{k_{n+1}-2}{\sum}}e^{sv})
\end{equation}
Considering (\ref{eq7}), (\ref{eq8}) and (\ref{eq9}) we get
\begin{equation}\label{eq10}
\frac{1}{\mathcal{Q}e^{\mathcal{Q}v}(1+\mathcal{Q}e^{\mathcal{Q}v})}\leq \frac{|I^{n+1}|}{|I^{n-1,n}_{i}|}\leq \frac{e^{v}}{2+e^{v}}
\end{equation}
and
\begin{equation}\label{eq11}
\frac{1}{(1+e^{v})\mathcal{Q}e^{\mathcal{Q}v}}\leq \frac{|I^{n,n+1}|}{|I^{n-1,n}_{i}|}\leq \frac{\mathcal{Q}e^{(\mathcal{Q}+1)v}}{1+e^{v}+\mathcal{Q}e^{(\mathcal{Q}+1)v}}.
\end{equation}
Denote by $C_{7}=[\mathcal{Q}e^{\mathcal{Q}v}(1+\mathcal{Q}e^{\mathcal{Q}v})]^{-1}$ and $C_{8}=\mathcal{Q}e^{\mathcal{Q}v}[1+\mathcal{Q}e^{\mathcal{Q}v}]^{-1}.$
Finally, if $k_{n+1}=1$ then $C_{7}|I^{n}|\leq |I^{n+1}|\leq |I^{n}|,$
if $k_{n+1}\geq 2$ then $C_{7}|I^{n}|\leq |I^{n+1}|\leq C_{8}|I^{n}|.$
\end{proof}
Therefore, every interval of $\textbf{D}_{n}$ contains at least
two interval of $\textbf{D}_{n+2},$ using this we get following
remark.
\begin{rem}\label{r1}
Let $I^{n+2}\in\textbf{D}_{n+2},$ $I^{n}\in\textbf{D}_{n}$ and $I^{n+2}\subset I^{n}.$
Then the following inequalities holds
\begin{equation}\label{12}
C_{7}^{2}|I^{n}|\leq |I^{n+2}|\leq C_{8}|I^{n}|.
\end{equation}
 \end{rem}
\section{Universal bounds to the barycentric coefficients}
Let $f$ be a P-homeomorphism with two break points $a,$ $b$ which is does not lie on the
same orbit and with irrational rotation number of bounded type.
Consider partition $\textbf{D}_{n}=\textbf{D}_{n}(a, f).$
Denote by $\mathcal{I}^{n}(b)$ the interval of $\textbf{D}_{n}$ that contains the point $b.$
In the following discussion we have to compare
different intervals. Let $[\alpha, \beta]$ be an interval in $S^{1}$ and $\gamma\in[\alpha, \beta].$
\emph{The barycentric coefficient} of $\gamma$ in $[\alpha, \beta]$ is the ratio
$$
\mathfrak{B}(\gamma; [\alpha, \beta])=\frac{|[\alpha, \gamma]|}{|[\alpha, \beta]|}.
$$
A \emph{universal bound} for $f$ is a constant that does not depend on $n$ and
does not depend on point. Now we show that there exists a subsequence $(n_{s})$ of $\mathbb{N}$ such that
the barycentric coefficient of the point $b$ in $\mathcal{I}^{n_s}=\mathcal{I}^{n_s}(b)$
is universal bounded in $(0, 1).$
\begin{prop}\label{proposition 4.4}
Let $f$ be a P-homeomorphism with two break points $a,$ $b$ which is does not lie on the
same orbit and with irrational rotation number of bounded type.
 Then there exists a subsequence $n_{s}\in\mathbb{N}$ such that for all $n_{s}$ holds the
 following inequalities
 $C_{7}^{2}\leq\mathfrak{B}(b; \mathcal{I}^{n_s})\leq 1-C_{7}^{2}.$
\end{prop}
\begin{proof}
We argue by contradiction and suppose that there exists a natural number
$n_{0}\geq 1$ such that for all $n\geq n_{0}$ hold
\begin{equation}\label{eq13}
\mathfrak{B}(b; \mathcal{I}^{n})<C_{7}^{2}\,\,\,\,\,\,\,
\text{or}\,\,\,\,\,\,\,\, \mathfrak{B}(b; \mathcal{I}^{n})>1-C_{7}^{2}.
\end{equation}
This inequalities ensure that, if the interval
$\mathcal{I}^{n}$ partitioned into at least two intervals belonging to
$\textbf{D}_{n+1},$ then the point $b$
is always in one of the extremal intervals
($I^{n+1}_{f}$-first or $I^{n+1}_{l}$-last) among
the $\textbf{D}_{n+1}$ intervals that is contained $\mathcal{I}^{n}.$
Actually, if $b$ belongs to none of the two extremal intervals of
$\textbf{D}_{n+1}$ which are contained in $\mathcal{I}^{n},$ then the point
$b$ is separated from left edge of $\mathcal{I}^{n}$
by at least $|I^{n+1}_{f}|$ distance  and  from right edge of $\mathcal{I}^{n}$ by at least
$|I^{n+1}_{l}|$ distance. But, using property \ref{proper4},  we have gotten be the following
inequalities
\begin{equation}\label{eq14}
 C_{7}^{2} < C_{7}\leq \frac{|I^{n+1}_{f}|}{|\mathcal{I}^{n}|}\leq\mathfrak{B}(b; \mathcal{I}^{n})
 \leq 1-\frac{|I^{n+1}_{l}|}{|\mathcal{I}^{n}|} < 1-C_{7}\leq 1-C_{7}^{2}.
\end{equation}
If the interval $\mathcal{I}^{n}$ is member of partition $\textbf{D}_{n+1}$, then it is 	
surely partitioned into at least two intervals belonging to $\textbf{D}_{n+2}.$ In this case
the point $b$ also lies in one of the extremal
$I^{n+2}_{f}$ or $I^{n+2}_{l}$ intervals among
the $\textbf{D}_{n+2}$ intervals that is contained $\mathcal{I}^{n}.$
Otherwise using remark \ref{r1} we have gotten be the following
inequalities
\begin{equation}\label{eq15}
 C_{7}^{2} \leq \frac{|I^{n+2}_{f}|}{|\mathcal{I}^{n}|}\leq\mathfrak{B}(b; \mathcal{I}^{n})
 \leq 1-\frac{|I^{n+2}_{l}|}{|\mathcal{I}^{n}|} \leq 1-C_{7}^{2}.
\end{equation}
So, both inequalities (\ref{eq14}) and (\ref{eq15}) are contradiction
to  (\ref{eq13}).
Continue, the point $b$ cannot be indefinitely in the first interval
of the partition $\textbf{D}_{n+1}$ which is contained in
$\mathcal{I}^{n}.$ That is, for all positive integers $n',$ there
exists $n\geq n'$ such that $\mathcal{I}^{n+1}$ is not the first
intervals  of the partition
$\textbf{D}_{n+1}$ that are contained $\mathcal{I}^{n}.$
If not, since the length of intervals of partition $\textbf{D}_{n}$
 tends to zero when $n$ goes to infinity, the point
 $b$ would be arbitrary near the common left extremity
 of all  $\mathcal{I}^{n},$ for all $n\geq n'.$
 Therefore,  $b$ would equal this point which is the left
 end point of $\mathcal{I}^{n'}$ and hence an iterate of the
 point $a,$ which contradicts the hypothesis that the break points
 $a$ and $b$ of $f$  have disjoint orbits.
 In a like manner, we can show that
 the point $b$ cannot be indefinitely in the last interval
of the partition $\textbf{D}_{n+1}$ which is contained in
$\mathcal{I}^{n}.$

So, there exists a subsequence on $n_s \in \mathbb{N},$ such that the point $b$
is in the first of the $\textbf{D}_{n_s+1}$ intervals that are contained in $\mathcal{I}^{n_s}$ and in the last of the $\textbf{D}_{n_s+2}$ intervals that are contained $\mathcal{I}^{n_s+1}$. In this situation, the point $b$ is separated from left edge of $\mathcal{I}^{n_s}$ by at least $|I^{n_s+2}_{f}|$ distance  and  from right edge of $\mathcal{I}^{n_s}$ by at least $|I^{n_s+2}_{l}|$ distance. By remark \ref{r1} we get
\begin{equation}\label{eq16}
 C_{7}^{2} \leq \frac{|I^{n_s+2}_{f}|}{|\mathcal{I}^{n_s}|}\leq\mathfrak{B}(b; \mathcal{I}^{n_s})
 \leq 1-\frac{|I^{n_s+2}_{l}|}{|\mathcal{I}^{n_s}|} \leq 1-C_{7}^{2}.
\end{equation}
\end{proof}
\section{Universal bounds to the consecutive break points of $f^{q_{m}}$}
In this section we consider a P-homeomorphism $f$ with two break points $a,$ $b$ which is does not lie on the
same orbit. It is clear that the map $f^{q_{m}}$ has for break points $a_{j}=f^{-j}(a)$ and $b_{j}=f^{-j}(b)$ with $j=0,1,...,q_{m}-1.$
If we consider $m$-th dynamical partition $\textbf{D}_{m}=\textbf{D}_{m}(a, f)$ of the break point $a,$
then the break points $a_{j},$ $0\leq j< q_{m}$ of $f^{q_{m}}$ lies on the endpoints of intervals
$I^{m-1,m}_{j}$ of $\textbf{D}_{m}$ for every natural number $m.$
Eventually, we must study behavior of only the break points $b_{j},$ $0\leq j< q_{m}$ of $f^{q_{m}}.$
Thats why we have introduced the partition $\textbf{D}_{m}.$
  Now we consider $m$-th dynamical partition $\textbf{P}_{m}=\textbf{P}_{m}(a, f)$ of the break point $a.$
  The break points $a_{j},$ $0\leq j< q_{m}$ of $f^{q_{m}}$  can not be  the endpoints of intervals
of $\textbf{P}_{m}$ and the atoms of $\textbf{P}_{m}$ containing the points $a_{j}$ in its interior  denoted dy $\Delta^{m}(a_{j})$.
  Next theorem show that if for some $m=n_{s}\geq n_{0}$ the barycentric coefficient of the point
  $b$ in  $\mathcal{I}^{m}$ is universally bounded
then for such $m$ there exists the universal bounds between any two consecutive break points of $f^{q_{m}}$
i.e. any two consecutive break points of $f^{q_{m}}$ can not be very close to each other.
More precisely:

\begin{theo}\label{proposition 4.5} Let $f$ be a P-homeomorphism with break points $a$ and $b.$
 Suppose for some $m\geq n_{0}$ the barycentric coefficient of the point $b$ in $\mathcal{I}^{m}=\mathcal{I}^{m}(b)\in\textbf{D}_{m}$
 is universally bounded with constants $0<C_{7}^{2}<1-C_{7}^{2}<1,$ i.e.
 $C_{7}^{2}\leq\mathfrak{B}(b; \mathcal{I}^{m})\leq 1-C_{7}^{2}.$
 Then there exists a natural number $l=l(f)$ such that:
\begin{description}
  \item[(i)] for all  $0\leq j<q_{m}$ the interior of the atom $\Delta^{m+l}(a_{j})\in \textbf{P}_{m+l}$ contains only one break point of $f^{q_{m}}:$ the point $a_{j},$ hence $Df^{q_{m}}$ is continuous on each component of $\Delta^{m+l}(a_{j})\setminus \{a_{j}\};$
  \item[(ii)]for all  $0\leq j<q_{m}$ the barycentric coefficients of the points $a_{j}$ in $\Delta^{m+l}(a_{j})$ is universally bounded  with constants $0<C_{6}<1-C_{6}<1,$ i.e.
      $$
      C_{6}\leq\mathfrak{B}(a_{j}; \Delta^{m+l}(a_{j}))\leq 1-C_{6}.
      $$
\end{description}
\end{theo}
\begin{proof}
To prove first the statement of this theorem  first we find consecutive break points of $f^{q_{m}}$
and then we estimate  distance between consecutive break points of $f^{q_{m}}.$
W.l.o.g. choose $m$ to be odd. Then $m$-th and $m-1$-th generator intervals are the following form:
 $\Delta^{m}_{0}(a)=[f^{q_{m}}(a), a]$ and $\Delta^{m-1}_{0}(a)=[a, f^{q_{m-1}}(a)].$
By definition of dynamical partition $\textbf{D}_{m}=\textbf{D}_{m}(a, f)$
 the interval $\mathcal{I}^{m}(b)\in \textbf{D}_{m}$ is either
 an interval  $I^{m}_{k_0}=f^{k_{0}}[a, f^{-q_{m}}(a)]$  for some $0\leq k_{0}<q_{m}+q_{m-1}$ or
an interval  $I^{m-1,m}_{j_{0}}=f^{j_{0}}[f^{-q_{m}}(a),  f^{q_{m-1}}(a)]$ for some $0\leq j_{0}<q_{m}.$
First we suppose $\mathcal{I}^{m}(b)=f^{k_{0}}[a, f^{-q_{m}}(a)].$
It is easy to see that if $0\leq k_{0}<q_{m}$ then
the consecutive break points of $f^{q_{m}}$ are
following:
\begin{eqnarray}\label{eq17a}
\left\{
  \begin{array}{ll}
    f^{-k_{0}+i}(b), f^{-q_{m}+i}(a)\in\Delta^{m-1}_{i}, & \hbox{$1\leq i\leq k_{0}$;} \\
    f^{-q_{m}+i}(a), f^{-q_{m}+i-k_{0}}(b)\in\Delta^{m-1}_{i}, & \hbox{$k_{0}< i\leq q_{m}$.}
  \end{array}
\right.
\end{eqnarray}
From relations (\ref{eq17a}) implies that if $1\leq i\leq k_{0}$ then between break point $f^{-q_{m}+i}(a)$ of $f^{q_{m}}$
 and right endpoint $f^{q_{m-1}+j}(a)$ of  $\Delta^{m-1}_{i}$ can not lies another break points of  $f^{q_{m}}.$
 Similarly, if $k_{0}< i\leq q_{m}$ then between break point $f^{-q_{m}+i}(a)$ of $f^{q_{m}}$
 and left endpoint $f^{j}(a)$ of  $\Delta^{m-1}_{i}$ can not lies another break points of  $f^{q_{m}}.$
 By definition of dynamical partition $\textbf{D}_{m}$ we have
 $I^{m-1,m}_{i}=f^{i}[f^{-q_{m}}(a), f^{q_{m-1}}(a)]$ and $I^{m}_{i}=f^{i}[a, f^{-q_{m}}(a)].$
 Using inequalities (\ref{eq5}) we get $|I^{m-1,m}_{i}|\geq C_{6}|\Delta^{m-1}_{i}|,$ $1\leq i\leq k_{0}$ and
 $|I^{m}_{i}|\geq (1+e^{v})^{-1}|\Delta^{m-1}_{i}|,$ $k_{0}< i\leq q_{m}.$
Now we estimate the distance between this consecutive break points of $f^{q_{m}}.$
By assumption of theorem \ref{proposition 4.5}
\begin{eqnarray}\label{eq17}
 C_{7}^{2}\leq \mathfrak{B}(b; \mathcal{I}^{m})=\frac{|[f^{k_{0}}(a), b]|}{|[f^{k_{0}}(a), f^{-q_{m}+k_{0}}(a)]|}\leq 1-C_{7}^{2},
\end{eqnarray}
from this imply
\begin{eqnarray}\label{eq18}
 C_{7}^{2}\leq 1-\mathfrak{B}(b; \mathcal{I}^{m})=\frac{|[b, f^{-q_{m}+k_{0}}(a)]|}{|[f^{k_{0}}(a), f^{-q_{m}+k_{0}}(a)]|}\leq 1-C_{7}^{2}.
\end{eqnarray}
So, by the Denjoy estimate, we get
\begin{eqnarray}\label{eq19}
 e^{-2v}C_{7}^{2}\leq \frac{|f^{-q_{m}}[f^{k_{0}}(a), b]|}{|f^{-q_{m}}[f^{k_{0}}(a), f^{-q_{m}+k_{0}}(a)]|}\leq e^{2v}(1-C_{7}^{2}).
\end{eqnarray}
Using Finzi inequality (generalized Finzi inequality) to the inequalities (\ref{eq18}) and (\ref{eq19})
we can show that the distance  between consecutive
break points $f^{-k_{0}+i}(b)$ and $f^{-q_{m}+i}(a),$ $1\leq i \leq k_{0}$ of $f^{q_{m}}$ greater then
$C_{7}^{2}e^{-v}|f^{i}[a, f^{-q_{m}}(a)]|=C_{7}^{2}e^{-v}|I^{m}_{i}|.$ Similarly, the  distance  between
consecutive break points $f^{-q_{m}+i}(a)$ and $f^{-q_{m}+i-k_{0}}(b),$ $k_{0}< i\leq q_{m}$ of $f^{q_{m}}$
greater then $C_{7}^{2}e^{-2v}|f^{-i}[f^{q_{m}}(a), f^{-2q_{m}}(a)]|=C_{7}^{2}e^{-2v}|f^{-q_{m}}(I^{m}_{i})|.$
Using inequalities (\ref{eq5}) together with Denjoy inequalities it is easy to see that
the distance between  consecutive break points $f^{-k_{0}+i}(b)$ and
$f^{-q_{m}+i}(a),$ $1\leq i\leq k_{0}$ of $f^{q_{m}}$ greater then
$C_{7}^{2}C_{6}e^{-v}|\Delta^{m-1}_{i}|$ and the distance  between
 consecutive break points $f^{-q_{m}+i}(a)$ and
$f^{-q_{m}+i-k_{0}}(b),$ $k_{0}< i\leq q_{m}$ of $f^{q_{m}}$ greater then
$C_{7}^{2}C_{6}e^{-3v}|\Delta^{m-1}_{i}|.$

Now we consider the case  $q_{m}\leq k_{0}<q_{m}+q_{m-1}.$
In this case the consecutive break points of $f^{q_{m}}$ are
following:
\begin{eqnarray}\label{eq19a}
\left\{
  \begin{array}{ll}
    f^{-(k_{0}-q_{m})+i}(b), f^{-q_{m}+i}(a)\in\Delta^{m}_{i}\cup\Delta^{m-1}_{i}, & \hbox{$1\leq i \leq k_{0}-q_{m}$;} \\
    f^{-k_{0}+i}(b), f^{-q_{m}+i}(a)\in\Delta^{m-1}_{i}, & \hbox{$k_{0}-q_{m}< i\leq q_{m}$.}
  \end{array}
\right.
\end{eqnarray}
The same manner as above we can show that the distance  between
of the break points $ f^{-(k_{0}-q_{m})+i}(b)$ and $f^{-q_{m}+i}(a),$ $1\leq i\leq k_{0}$ of $f^{q_{m}}$ greater then
$(C_{7}^{2}+e^{-v})C_{2}e^{-v}|\Delta^{m-1}_{i}|$ and the  distance  between of the
break points $f^{-k_{0}+i}(b)$ and $f^{-q_{m}+i}(a)\,$ $k_{0}-q_{m}< i\leq q_{m}$ of $f^{q_{m}}$ greater then
$C_{7}^{2}C_{2}e^{-3v}|\Delta^{m-1}_{i}|.$
Moreover, for all  $1\leq i\leq q_{m}$  between break point $f^{-q_{m}+i}(a)$ of $f^{q_{m}}$
 and right endpoint $f^{q_{m-1}+j}(a)$ of  $\Delta^{m-1}_{i}$ can not lies another break points of  $f^{q_{m}}$
 and $|I^{m-1,m}_{i}|\geq C_{6}|\Delta^{m-1}_{i}|$ for all $1\leq i\leq q_{m}.$

Now we suppose $\mathcal{I}^{m}(b)=f^{j_{0}}[f^{-q_{m}}(a),  f^{q_{m-1}}(a)]$ for some $0\leq j_{0}<q_{m}.$
In this case to  find the  consecutive break points of $f^{q_{m}}$ we consider two cases:
$b \in f^{j_{0}}[f^{-q_{m}}(a),  f^{q_{m}+q_{m-1}}(a)]\subset \mathcal{I}^{m}(b)$ and $b \in f^{j_{0}}[f^{q_{m}+q_{m+1}}(a),  f^{q_{m-1}}(a)]\subset \mathcal{I}^{m}(b).$ If $b \in f^{j_{0}}[f^{-q_{m}}(a),  f^{q_{m}+q_{m-1}}(a)]$ then consecutive break points of $f^{q_{m}}$
are following:
\begin{eqnarray}\label{eq19b}
\left\{
  \begin{array}{ll}
    f^{-q_{m}+i}(a), f^{-j_{0}+i}(b) \in\Delta^{m-1}_{i}, & \hbox{$1\leq i\leq j_{0}$;} \\
    f^{-q_{m}+i}(a), f^{-q_{m}+i-j_{0}}(b)\in\Delta^{m-1}_{i}, & \hbox{$j_{0}< i\leq q_{m}$.}
  \end{array}
\right.
\end{eqnarray}
It is clear that, for all $1\leq i\leq q_{m}$  between break point $f^{-q_{m}+i}(a)$ of $f^{q_{m}}$
 and left endpoint $f^{j}(a)$ of  $\Delta^{m-1}_{i}$ can not lies another break points of  $f^{q_{m}}$
 and  $|I^{m}_{i}|\geq (1+e^{v})^{-1}|\Delta^{m-1}_{i}|.$
 Now we estimate the distance between consecutive
 break points of $f^{q_{m}}.$
By assumption of theorem \ref{proposition 4.5} the barycentric coefficient
of the point $b$ in $\mathcal{I}^{m}$ is universally bounded i.e.
\begin{eqnarray}\label{eq20}
 C_{7}^{2}\leq \mathfrak{B}(b; \mathcal{I}^{m})=\frac{|[f^{-q_{m}+j_{0}}(a), b]|}{|[f^{-q_{m}+j_{0}}(a), f^{q_{m-1}+j_{0}}(a)]|}\leq 1-C_{7}^{2}.
\end{eqnarray}
By above notation $I^{m-1,m}_{j_{0}}=[f^{-q_{m}+j_{0}}(a), f^{q_{m-1}+j_{0}}(a)].$
Using inequalities (\ref{eq5}) and (\ref{eq20}) together with  Finzi inequality  we get
$C_{7}^{2}e^{-v}(1+e^{v})^{-1}|\Delta^{m-1}_{i}|\leq |f^{-q_{m}+i}(a), f^{-j_{0}+i}(b)|$ for all $1\leq i\leq j_{0}.$
The same manner as above we can show that
$C_{7}^{2}(1+e^{v})^{-1}e^{-v}|\Delta^{m-1}_{i}|\leq |f^{-q_{m}+i}(a), f^{-j_{0}+i}(b)|<|f^{-q_{m}+i}(a), f^{-q_{m}+i-j_{0}}(b)|$
 for all $j_{0}\leq i\leq q_{m}.$

 Now we consider the case
 $b \in f^{j_{0}}[f^{q_{m}+q_{m-1}}(a),  f^{q_{m-1}}(a)]\subset \mathcal{I}^{m}(b).$
 In this case if $0\leq j_{0}<q_{m}-q_{m-1}$ then consecutive break points of $f^{q_{m}}$ are
following:
\begin{eqnarray}\label{eq20a}
\left\{
  \begin{array}{ll}
    f^{-j_{0}-q_{m-1}+i}(b), f^{-q_{m}+i}(a)\in\Delta^{m-1}_{i}\cup\Delta^{m}_{i}, & \hbox{$1\leq i\leq j_{0}+q_{m-1}$;} \\
    f^{-q_{m}-j_{0}-q_{m-1}+i}(b), f^{-q_{m}+i}(a)\in\Delta^{m-1}_{i}, & \hbox{$j_{0}+q_{m-1}< i\leq q_{m}$.}
  \end{array}
\right.
\end{eqnarray}
Similarly we can show that
  $C^{2}_{7}(1+e^{v})^{-1}e^{-3v}|\Delta^{m-1}_{i}|\leq|f^{-j_{0}-q_{m-1}+i}(b), f^{-q_{m}+i}(a)|$ for all
  $1\leq i\leq j_{0}+q_{m-1}$ and $C^{2}_{7}(1+e^{v})^{-1}e^{-4v}|\Delta^{m-1}_{i}|\leq|f^{-q_{m}-j_{0}-q_{m-1}+i}(b), f^{-q_{m}+i}(a)|$
 for all $j_{0}+q_{m-1}< i\leq q_{m}.$
If $q_{m}-q_{m-1}\leq j_{0}<q_{m}$ then consecutive break points of $f^{q_{m}}$ are
following:
\begin{eqnarray}\label{eq20b}
\left\{
  \begin{array}{ll}
    f^{-q_{m}+i}(a), f^{-j_{0}+i}(b)\in\Delta^{m-1}_{i}, & \hbox{$1\leq i \leq j_{0}$;} \\
    f^{-q_{m}+i}(a), f^{-q_{m}-j_{0}+i}(b)\in\Delta^{m-1}_{i}\cup\Delta^{m}_{q_{m-1}-q_{m}+i}, & \hbox{$j_{0}< i\leq q_{m}$.}
  \end{array}
\right.
\end{eqnarray}
In this case
 $C^{2}_{7}(1+e^{v})^{-1}e^{-v}|\Delta^{m-1}_{i}|\leq |f^{-q_{m}+i}(a), f^{-j_{0}+i}(b)|$ for all $1\leq i \leq j_{0}$
 and $C^{2}_{7}(1+e^{v})^{-1}e^{-v}|\Delta^{m-1}_{i}|\leq|f^{-q_{m}+i}(a), f^{-q_{m}-j_{0}+i}(b)|$ for all $j_{0}< i\leq q_{m}$.
  Using above concepts it is easy to see that
 distance between the consecutive break points of $f^{q_{m}}$ greater then $\min \{C^{2}_{7}(1+e^{v})^{-1}e^{-4v}, C^{2}_{7}C_{2}e^{-3v}\}|\Delta^{m-1}_{i}|$ for all
 $0\leq i\leq q_{m}.$ By remark \ref{inter nisbat va expon rem} there exist such
  $l\in \mathbb{N}$ that hold this $|\Delta^{m+l}(a_{j})|\leq (1+e^{v})e^{3v}\lambda^{l}|\Delta^{m}(a_{j})|$ inequality. If we take a natural
   number $l$ such that $2(1+e^{v})e^{3v}\lambda^{l}<\min \{C^{2}_{7}(1+e^{v})^{-1}e^{-4v}, C^{2}_{7}C_{2}e^{-3v}\}$ then the interior of the atom $\Delta^{m+l}(a_{j})$ contains only one break point of $f^{q_{m}},$ because distance between the nearest break points of $f^{q_{m}}$ greater then $|\Delta^{m+l}(a_{j})|.$

Now we prove the second assertion of theorem \ref{proposition 4.5}.
Using property of dynamical partition $\textbf{P}_{m+l}$  for
every $0< j\leq q_{m}$ we can written explicit form of the intervals $\Delta^{m+l}(a_{j})$
as the following form:
 \begin{description}
  \item[-] $\Delta^{m+l}(a_{j})=[f^{q_{m+l}-j}(a), \,\, f^{q_{m+l}+q_{m+l-1}-j}(a)],$ \,\,\  if $l$ is even,
  \item[-]$\Delta^{m+l}(a_{j})=[f^{q_{m+l}+q_{m+l-1}-j}(a), \,\, f^{q_{m+l}-j}(a)],$ \,\,\  if $l$ is odd.
\end{description}
If $l$ is even then the barycentric coefficient of point $a_{j}$
in $\Delta^{m+l}(a_{j})$ is equal to the following ratio:
$$
\mathfrak{B}(a_{j}; \Delta^{m+l}(a_{j}))=\frac{|[f^{q_{m+l}-j}(a), a_{j}]|}
{|[f^{q_{m+l}-j}(a), f^{q_{m+l}+q_{m+l-1}-j}(a)]|}=
\frac{|[f^{q_{m+l}-j}(a), f^{-j}(a)]|}
{|[f^{q_{m+l}-j}(a), f^{q_{m+l}+q_{m+l-1}-j}(a)]|}
$$
In the case $l$ is odd then the barycentric coefficient of point $a_{j}$ is equal to the
following ratio:
$$
\mathfrak{B}(a_{j}; \Delta^{m+l}(a_{j}))=\frac{|[f^{q_{m+l}+q_{m+l-1}-j}(a), a_{j}]|}
{|[f^{q_{m+l}+q_{m+l-1}-j}(a), f^{q_{m+l}-j}(a)]|}=
\frac{|[f^{q_{m+l}+q_{m+l-1}-j}(a), f^{-j}(a)]|}
{|[f^{q_{m+l}+q_{m+l-1}-j}(a), f^{q_{m+l}-j}(a)]|}.
$$
Let us take change of variable $z=f^{q_{m+l}-j}(a),$ then
$$
\mathfrak{B}(f^{-q_{m+l}}(z); \Delta^{m+l}(f^{-q_{m+l}}(z)))=
\frac{|[z, f^{-q_{m+l}}(z)]|}
{|[z, f^{q_{m+l-1}}(z)]|}=\frac{|I^{n+l}_{0}(z)|}{|\Delta^{n+l-1}_{0}(z)|},
$$
if $l$ is even and if $l$ is odd, then
$$
\mathfrak{B}(f^{-q_{m+l}}(z); \Delta^{m+l}(f^{-q_{m+l}}(z)))=
\frac{|I^{n+l-1, n+l}_{0}(z)|}{|\Delta^{n+l-1}_{0}(z)|}.
$$
Using inequalities (\ref{eq5}) it is easy to see that the
following inequalities hold for both cases of $l$
$$
C_{6}\leq\mathfrak{B}(f^{-q_{m+l}}(z); \Delta^{m+l}(f^{-q_{m+l}}(z)))\leq 1-C_{6}.
$$
\end{proof}

\section {Estimates for differences of $\log Df^{q_{n}}$}
Let  $f_{i},$ $i=1,2$ be circle homeomorphisms with two break points $ a_{i},b_{i}$
 satisfying the conditions $(1)-(5)$ of theorem \ref{ADM1}. We introduce the following
function on the circle
$$
F_{n}(x)=\frac{Df_{2}^{q_n}(h(x))}{Df_{1}^{q_n}(x)}=\frac{Df_{2}(h(x))\cdot Df_{2}(f_{2}(h(x)))\cdot\cdot\cdot Df_{2}(f^{q_{n}-1}_{2}(h(x)))}
{Df_{1}(x)\cdot Df_{1}(f_{1}(x))\cdot\cdot\cdot Df_{1}(f^{q_{n}-1}_{1}(x))}.
$$
The map $F_{n}$ has for jump points (i.e. the map $F_{n}$ has jump)
$a^{1}_{k}=f_{1}^{-k}(a_{1}),$  $b^{1}_{k}=f_{1}^{-k}(b_{1})$
and $a^{2}_{k}=h^{-1}(f_{2}^{-k}(a_{2})),$  $b^{2}_{k}=h^{-1}(f_{2}^{-k}(b_{2}))$ with
$0\leq k<q_{n}.$ To prove the theorem \ref{ADM1} we will consider the following two cases:
$$
\textbf{either} \,\,\, \mu_{1}[a_{1}, b_{1}]=\mu_{2}[a_{2}, b_{2}]
\,\, \text{and}\,\,\
 \textbf{or} \,\,\ \mu_{1}[a_{1}, b_{1}]\neq \mu_{2}[a_{2}, b_{2}],
 \,\,\,
 $$
 where $ \mu_{i}$ is an invariant probability measure of $f_{i}, i=1,2.$
 Consider first the case $\mu_{1}[a_{1}, b_{1}]=\mu_{2}[a_{2}, b_{2}]$. Since conjugation map $h$ is unique up to additive constant we choose $h$ such that $h(a_{1})=a_{2},$ then by assumption $\mu_{1}[a_{1}, b_{1}]=\mu_{2}[a_{2}, b_{2}]$ implies that $h(b_{1})=b_{2}.$ Using this we get $h(f_{1}^{-k}(a_{1}))=f_{2}^{-k}(a_{2})$ and $h(f_{1}^{-k}(b_{1}))=f_{2}^{-k}(b_{2})$  for all $0\leq k<q_{n}.$
 It is easy to see the jump  points of $F_{n}$ are $a^{1}_{k}=f_{1}^{-k}(a_{1}),$  $b^{1}_{k}=f_{1}^{-k}(b_{1})$ $0\leq k<q_{n}$
 i.e. the jump points of $F_{n}$
 composed  only of the break points of  $f_{1}^{q_{n}}.$
The jumps of $F_{n}$ at these
points are following:
 $$
 \sigma_{F_{n}}(a^{1}_{k})=\frac{\sigma_{f^{q_{n}}_{2}}(h(a^{1}_{k}))}
 {\sigma_{f^{q_{n}}_{1}}(a^{1}_{k})}=\frac{\sigma_{f_{2}}(h(a_{1}))}
 {\sigma_{f_{1}}(a_{1})}=\frac{\sigma_{f_{2}}(a_{2})}
 {\sigma_{f_{1}}(a_{1})},
  $$
$$
 \sigma_{F_{n}}(b^{1}_{k})=\frac{\sigma_{f^{q_{n}}_{2}}(h(b^{1}_{k}))}
 {\sigma_{f^{q_{n}}_{1}}(b^{1}_{k})}=\frac{\sigma_{f_{2}}(h(b_{1}))}
 {\sigma_{f_{1}}(b_{1})}=\frac{\sigma_{f_{2}}(b_{2})}
 {\sigma_{f_{1}}(b_{1})},
  $$
and by assumption theorem \ref{ADM1} implies that
$
\sigma_{F_{n}}(a^{1}_{k})\neq \sigma_{F_{n}}(b^{1}_{k}),$ and
$
\sigma_{F_{n}}(a^{1}_{k})\times  \sigma_{F_{n}}(b^{1}_{k})=1$ for all $0\leq k<q_{n}.$
Denote by $3\delta_{0}=|\log \sigma_{f_{2}}(a_{2})-\log \sigma_{f_{1}}(a_{1})|>0.$
  Apply theorem \ref{proposition 4.5} to the function $f_{1}$ we can find subsequence
   $n_{s}=n_{s}(f_{1})\in \mathbb{N}$ such that break points of  $f_{1}^{q_{n_{s}}}$  far from each other.
  Let  $l=l(f_{1})$ be the natural number which is defined in theorem \ref{proposition 4.5}. The following proposition is formulated for a suitable subsequence $n_{s}\in \mathbb{N}$ and natural number $l.$
\begin{prop}\label{proposition 4.6}
Assume the homeomorphisms  $f_{i}, i=1,2 $ satisfy the conditions of
 Theorem \ref{ADM1}.
  Then there exists a natural number $l_{0}=l_{0}(f_{1}, f_{2})$ such that $l_{0}\geq l$ and  for all $0\leq k < q_{n_{s}},$  on one of the two connected components of
$\Delta_{f_{1}}^{n_{s}+l_{0}}(a^{1}_{k})\backslash\{a^{1}_{k}\},$
the following inequality holds: $$|\log Df^{q_{n_{s}}}_{2}(h(x))- \log Df^{q_{n_{s}}}_{1}(x)|\geq
\delta_{0}.$$
\end{prop}
\begin{proof}
Let us take a positive integer $\overline{l}$ such that $C_{5}(f_{1})\lambda_{1}^{\overline{l}/q}+C_{5}(f_{2})\lambda_{2}^{\overline{l}/q}\leq \frac{\delta_{0}}{2},$ where $C_{5}(f_{i})$ and $\lambda_{i}, i=1,2$  appropriate constants of $f_{i}, i=1,2$ which are satisfies
lemma \ref{universal estimates}.
Denote by $l_{0}=\max \{ \overline{l}, l \}.$  According to theorem  \ref{proposition 4.5} (i), the interior of the atom $\Delta^{n_{s}+l_{0}}_{f_1}(a^{1}_{k})$  contains only one jump point of $F_{n_{s}}$ the point $a_{k}$ hence, $F_{n_{s}}$ is continuous on each component  $\Delta^{n_{s}+l_{0}}_{f_1}(a^{1}_{k})\setminus \{a^{1}_k\}.$

If $|\log F_{n_{s}}(x)|\geq \delta_{0}$ on the left component of $\Delta^{n_{s}+l_{0}}_{f_1}(a^{1}_{k})\setminus \{a^{1}_k\},$ then we are done. If not, there exists at least one point $x$  on the left component of $\Delta^{n_{s}+l_{0}}_{f_1}(a^{1}_{k})\setminus \{a^{1}_k\}$ such that
\begin{eqnarray}\label{eq22}
|\log F_{n_{s}}(x)|<\delta_{0}.
\end{eqnarray}
Now, for any $y$ in the left component of $\Delta^{n_{s}+l_{0}}_{f_1}(a^{1}_{k})\setminus \{a^{1}_k\}$ we have
  \begin{eqnarray}\label{eq23}
|\log F_{n_{s}}(y)|\leq |\log F_{n_{s}}(x)|+|\log F_{n_{s}}(x)-\log F_{n_{s}}(y)|< \delta_{0}+\frac{\delta_{0}}{2}.
\end{eqnarray}
by lemma \ref{universal estimates} with $k=q_{n_{s}}$ and $l=l_{0}.$
Then in particular $|\log F_{n_{s}}(a^{1}_{k}-0)|<3\delta_{0}/2$ and
\begin{eqnarray}\label{eq24}
|\log F_{n_{s}}(a^{1}_{k}+0)|=|\log (\sigma_{F_{n_{s}}}(a^{1}_{k})F_{n_{s}}(a^{1}_{k}-0))|\geq \frac{3\delta_{0}}{2}.
\end{eqnarray}
Finally, for $y$ in  the right component of $\Delta^{n_{s}+l_{0}}_{f_1}(a^{1}_{k})\setminus \{a_k\},$ we have
\begin{eqnarray}\label{eq25}
|\log F_{n_{s}}(y)| \geq |\log (F_{n_{s}}(a^{1}_{k}+0))|-|\log (F_{n_{s}}(a^{1}_{k}+0))-\log F_{n_{s}}(y)|\geq \delta_{0}.
\end{eqnarray}
Hence on the right component of $\Delta^{n_{s}+l_{0}}_{f_1}(a^{1}_{k})\setminus \{a_k\},$ we have
$|\log F_{n_{s}}(y)| \geq \delta_{0}.$
\end{proof}
Now we consider second the case $\mu_{1}[a_{1}, b_{1}]\neq \mu_{2}[a_{2}, b_{2}]$. Without loss of generality, we can suppose that
 $\mu_{1}[a_{1}, b_{1}]<\mu_{2}[a_{2}, b_{2}]$ the opposite case can be handled similarly.  We choose $h$ such that $h(a_{1})=a_{2}.$
  Using this together with above inequality we get $h(b_{1})<b_{2}.$ Since $h$ is continuous and strictly increasing function, then there exist a unique point $c_{1}$ such that $b_{1}<c_{1}$ and $h(c_{1})=b_{2}.$ Using this it is easy to see the map $F_{n}$ has for jump points $a^{1}_{k}=f_{1}^{-k}(a_{1}),$ $b^{1}_{k}=f_{1}^{-k}(b_{1}),$  $c^{1}_{k}=f_{1}^{-k}(c_{1}).$
  In this case the jump points of $F_{n}$ obtained by adding some negative iterates of $c_{1}$ to the break points of  $f_{1}^{q_{n}}.$
  Moreover, the break points $c^{1}_{k}$ goes to the points $f_{2}^{-k}(b_{2})$ by $h$   i.e. $h(f_{1}^{-k}(c_{1}))=f_{2}^{-k}(b_{2}).$  The jumps of $F_{n}$ at these points are:
 $$
 \sigma_{F_{n}}(a^{1}_{k})=\frac{\sigma_{f^{q_{n}}_{2}}(h(a^{1}_{k}))}
 {\sigma_{f^{q_{n}}_{1}}(a^{1}_{k})}=\frac{\sigma_{f_{2}}(h(a_{1}))}
 {\sigma_{f_{1}}(a_{1})}=\frac{\sigma_{f_{2}}(a_{2})}
 {\sigma_{f_{1}}(a_{1})},
  $$
$$
 \sigma_{F_{n}}(b^{1}_{k})=(\sigma_{f^{q_{n}}_{1}}(b^{1}_{k}))^{-1}=
(\sigma_{f_{1}}(b_{1}))^{-1}, \,\,\,\
\sigma_{F_{n}}(c^{1}_{k})=\sigma_{f^{q_{n}}_{2}}(b^{2}_{k})=
\sigma_{f_{2}}(b_{2})
 $$
and by assumption theorem \ref{ADM1} implies that
$
\sigma_{F_{n}}(a^{1}_{k})\times \sigma_{F_{n}}(b^{1}_{s})\neq 1,$
$
\sigma_{F_{n}}(a^{1}_{k})\times  \sigma_{F_{n}}(c^{1}_{t}) \neq 1,$ and
$
\sigma_{F_{n}}(a^{1}_{k})\times \sigma_{F_{n}}(b^{1}_{s})\times
\sigma_{F_{n}}(c^{1}_{t})=1,
$
for all $0\leq k, s, t < q_{n}.$
  Let $n_{s}=n_{s}(f_{1})\in \mathbb{N}$ the subsequence such that break points of  $f_{1}^{q_{n_{s}}}$  far from each other.
  The main changes is in this case it is $F_{n_{s}}$ may not be continuous on
 one of the two connected components of
$\Delta_{f_{1}}^{n_{s}+l_{0}}(a^{1}_{k})\backslash\{a^{1}_{k}\}$ and one passes from every continuity interval of $F_{n_{s}}$ to the next one by multiplying $F_{n_{s}}$ by the jump at the common extremity of these two consecutive intervals.
Denote by $3\delta_{1}= \min \{|\log \sigma_{f_{2}}(a_{2})-\log\sigma_{f_{1}}(a_{1})+\log \sigma_{f_{2}}(b_{2})|, \, |\log \sigma_{f_{2}}(a_{2})-\log\sigma_{f_{1}}(a_{1})|\}.$ It is clear that $\delta_{1}$ is positive. Next, we will show that for any $0\leq k < q_{n_{s}}$ there exists a subinterval in $\Delta_{f_{1}}^{n_{s}+l_{0}}(a^{1}_{k})\}$ such that on this subinterval $|\log F_{n_{s}}|$ is  $\delta_{1}$- far from $0.$
\begin{prop}\label{proposition 4.7}
Assume the homeomorphisms  $f_{i}, i=1,2 $ satisfy the conditions of
Theorem \ref{ADM1}. Let $l_{0}$ be the constant which is defined
in proposition  \ref{proposition 4.6}.
\begin{description}
  \item[(i)] Then for any $0\leq k <q_{n_{s}}$ there exists a
 subinterval $I^{n_{s}}_{k}\subset \Delta_{f_{1}}^{n_{s}+l_{0}}(a^{1}_{k})$
  such that, on the interval $I^{n_{s}}_{k}$
the following inequality holds:
$$
|\log Df^{q_{n_{s}}}_{2}(h(x))- \log Df^{q_{n_{s}}}_{1}(x)|\geq
\delta_{1}.
$$

  \item[(ii)]There exists a universal constant $C_{9}=C_{9}(f_{1}, f_{2})>0$ such that the intervals $I^{n_{s}}_{k}$ and $\Delta_{f_{1}}^{n_{s}+l_{0}}(a^{1}_{k})$ are $C_{9}$- comparable.
\end{description}
\end{prop}
\begin{proof}
Let us take a natural number $\widehat{l}$ such that
$C_{5}(f_{1})\lambda_{1}^{\widehat{l}/q}+C_{5}(f_{2})\lambda_{2}^{\widehat{l}/q}\leq \frac{\delta_{1}}{3},
$ where $C_{5}(f_{i})$ and $\lambda_{i}, i=1,2$  appropriate constants of $f_{i}, i=1,2$ which are satisfies universal estimates.
Denote by $l_{1}=\max \{\widehat{l}, l_{0} \}.$  Let $\Delta_{f_{1}}^{n_{s}+l_{1}}(a^{1}_{k})$ be the atoms of $\textbf{P}_{n+l_{1}}(a, f_{1})$ containing the points $a^{1}_{k}$ in its interior. It is clear that $\Delta_{f_{1}}^{n_{s}+l_{1}}(a_{k})\subset \Delta_{f_{1}}^{n_{s}+l_{0}}(a^{1}_{k})$ and by remark \ref{inter nisbat va expon rem}
 there exists a constant $\kappa_{1}=\kappa(f_{1})>0$ such that $|\Delta_{f_{1}}^{n_{s}+l_{1}}(a^{1}_{k})|\geq e^{-3v_{1}}\kappa^{l_{1}-l_{0}}_{1} |\Delta_{f_{1}}^{n_{s}+l_{0}}(a^{1}_{k})|$ for all $0\leq k<q_{n_{s}},$ where $v_{1}=Var_{S^{1}}\log Df_{1}.$
Now we will construct an intervals $I^{n_{s}}_{k}$ which is comparable with
$\Delta_{f_{1}}^{n_{s}+l_{1}}(a^{1}_{k}).$ For this we define following sets
$B_{n_{s}}=\{c^{1}_{k}: \,\ c^{1}_{k}=f^{-k}_{1}(c_{1}),\,\ 0\leq k <q_{n_{s}}\},$ \,\
$G_{n_{s}}=\{k: 0\leq k <q_{n}, \Delta_{f_{1}}^{n_{s}+l_{1}}(a^{1}_{k})\cap B_{n_{s}}\neq \emptyset \}$ and  $\overline{G}_{n_{s}}$ is complement of
$G_{n_{s}}.$ Let for definiteness $G_{n_{s}}$ is non empty.
If $k \in \overline{G}_{n_{s}}$
 then the atom
$\Delta_{f_{1}}^{n_{s}+l_{1}}(a_{k})$ contains only one jump point of $F_{n_{s}}$ the point $a_{k}$ hence $F_{n_{s}}$ is continuous on each component $\Delta_{f_{1}}^{n_{s}+l_{1}}(a^{1}_{k})\setminus \{a^{1}_{k}\}.$ Since
$\delta_{1}\leq \delta_{0}$ and $l_{0}\leq l_{1},$ according to proposition \ref{proposition 4.6}, on one of the two connected components of
$\Delta_{f_{1}}^{n_{s}+l_{1}}(a^{1}_{k})\backslash\{a^{1}_{k}\},$
the following inequality holds $|\log F_{n_{s}}(x)|\geq \delta_{1}.$ Denote by $I_{k}^{n_{s}}$ a component of $\Delta_{f_{1}}^{n_{s}+l_{1}}(a^{1}_{k})\backslash\{a^{1}_{k}\},$ such holds last inequality.
 By theorem \ref{proposition 4.5}  (ii) there exists
$C_{6}=C_{6}(f_{1})$ such that
$$
|I_{k}^{n_{s}}|\geq C_{6}|\Delta_{f_{1}}^{n_{s}+l_{1}}(a^{1}_{k})|\geq C_{6}\kappa^{l_{1}-l_{0}}_{1}e^{-3v_{1}}|\Delta_{f_{1}}^{n_{s}+l_{0}}(a^{1}_{k})|
$$
Now, let $k \in G_{n_{s}}.$ Then the atom
$\Delta_{f_{1}}^{n_{s}+l_{1}}(a^{1}_{k})$ contains two jump points of $F_{n_{s}}$ the point $a^{1}_{k}$ and an element $c^{1}_{t}$ of $B_{n_{s}}.$ Hence $F_{n_{s}}$ is continuous on each component $\Delta_{f_{1}}^{n_{s}+l_{1}}(a^{1}_{k})\setminus \{a^{1}_{k}, c^{1}_{t}\}=L^{n_{s}}_{k}\cup M^{n_{s}}_{k} \cup R^{n_{s}}_{k}.$ Let for definiteness the point $c^{1}_{t}$ lie
 on the left hand side of the point $a^{1}_{k}.$ Then $R^{n_{s}}_{k}$  is right  component of  $\Delta_{f_{1}}^{n_{s}+l_{1}}(a^{1}_{k})\setminus \{a^{1}_{k}\}$ using above arguments the intervals $R^{n_{s}}_{k}$ and $\Delta_{f_{1}}^{n_{s}+l_{0}}(a^{1}_{k})$ are $C_{6}\kappa^{l_{1}-l_{0}}_{1}e^{-3v_{1}}$- comparable.
 If on the interval $R^{n_{s}}_{k}$ holds this
$|\log F_{n_{s}}(x)|< \delta_{1}$ inequality, then we take $I_{k}^{n_{s}}=R^{n_{s}}_{k}$ and desired result follows obviously.
If not, there exists at least one point $z_{0}$ in $R^{n_{s}}_{k}$ such that
 $|\log F_{n_{s}}(z_{0})|\leq \delta_{1}.$
 It is easy to see for any $x \in L^{n_{s}}_{k}$ and $y \in M^{n_{s}}_{k}$ holds this inequalities:
\begin{eqnarray}\label{eq26}
  |\log F_{n_{s}}(x)-\log F_{n_{s}}(z_{0})|\geq
  \end{eqnarray}
  $$
    \geq|\log \sigma_{F_{n_{s}}}(c^{1}_{t}) \sigma_{F_{n_{s}}}(a^{1}_{k})|-3(C_{5}(f_{1})\lambda_{1}^{l_{1}/q}
    +C_{5}(f_{2})\lambda_{2}^{l_{1}/q}),
    $$
\begin{eqnarray}\label{eq27}
  |\log F_{n_{s}}(y)-\log F_{n_{s}}(z_{0})|\geq
    |\log \sigma_{F_{n_{s}}}(a^{1}_{k})|-2(C_{5}(f_{1})\lambda_{1}^{l_{1}/q}
    +C_{5}(f_{2})\lambda_{2}^{l_{1}/q}).
    \end{eqnarray}
It is clear that the intervals $L^{n_{s}}_{k}\bigcup M^{n_{s}}_{k}$ and $R^{n_{s}}_{k}$ are two consecutive intervals of
 $D_{n_{s}+l_{1}}.$ By remark \ref{con inter}
 at least one of the intervals $L^{n_{s}}_{k}$ and $ M^{n_{s}}_{k}$ are $C_{2}C_{6}^{2}/2$ - comparable with $R^{n_{s}}_{k}.$
  First, let the interval $L^{n_{s}}_{k}$ be $C_{2}C_{6}^{2}/2$ - comparable with  $R^{n_{s}}_{k}.$
  Using (\ref{eq26}) for any point $x\in L^{n_{s}}_{k}$ we get
\begin{eqnarray}\label{eq28}
  |\log F_{n_{s}}(x)|\geq
    |\log \sigma_{F_{n_{s}}}(c^{1}_{t}) \sigma_{F_{n_{s}}}(a^{1}_{k})|-3(C_{5}(f_{1})\lambda_{1}^{l_{1}/q}
    +C_{5}(f_{2})\lambda_{2}^{l_{1}/q})-\delta_{1}.
\end{eqnarray}
Of the determine integer $l_{1}$ and positive $\delta_{1},$ implies
$|\log \sigma_{F_{n_{s}}}(c^{1}_{t}) \sigma_{F_{n_{s}}}(a^{1}_{k})|=|\log \sigma_{f_{2}}(a_{2})-\log\sigma_{f_{1}}(a_{1})+\log \sigma_{f_{2}}(b_{2})|\geq 3\delta_{1}$ and $3(C_{5}(f_{1})\lambda_{1}^{l_{1}/q}+C_{5}(f_{2})\lambda_{2}^{l_{1}/q})\leq \delta_{1}.$ The last two equations together with (\ref{eq28}) imply
$|\log F_{n_{s}}(x)|\geq \delta_{1}$ for any point $x\in L^{n_{s}}_{k}.$
If we take $I_{k}^{n_{s}}=L^{n_{s}}_{k}$ then the intervals $I_{k}^{n_{s}}$
and $\Delta_{f_{1}}^{n_{s}+l_{0}}(a^{1}_{k})$ are $C_{9}=2^{-1}C_{2}C_{6}^{3}\kappa^{l_{1}-l_{0}}_{1}e^{-3v_{1}}$ - comparable.
Secondly, let the interval $M^{n_{s}}_{k}$ be $C_{2}C_{6}^{2}/2$ - comparable with $R^{n_{s}}_{k}.$
Similarly we can show that
$|\log F_{n_{s}}(y)|\geq \frac{4\delta_{1}}{3}$ for any point $y\in M^{n_{s}}_{k}.$ In this case, if we take $I_{k}^{n_{s}}=M^{n_{s}}_{k}$
again the intervals $I_{k}^{n_{s}}$
and $\Delta_{f_{1}}^{n_{s}+l_{0}}(a^{1}_{k})$ are $C_{9}$ - comparable.
\end{proof}
\section{Proof of main theorem}
\begin{proof}Assume that the homeomorphisms  $f_{i}, i=1,2 $ satisfy the conditions of theorem \ref{ADM1}.
By lemma \ref{Zaruriy} the conjugation map $h$ between $f_{1}$ and $f_{2}$ is either absolutely continuous
or singular function. Suppose  $h$  is  absolutely continuous function. Then by theorem \ref{theor 3.1.} for all $\epsilon >0$
 there exists such natural number  $n_{0}$ such that for $n>n_{0}$ holds this
\begin{eqnarray}\label{eq29}
\ell(x: \, |\log Df^{q_{n}}_{2}(h(x))- \log Df^{q_{n}}_{1}(x)|\geq
\delta)<\epsilon
\end{eqnarray}
inequality for any $\delta>0.$
First, we assume $\mu_{1}[a_{1}, b_{1}]=\mu_{2}[a_{2}, b_{2}]$ where $\mu_{i}, \,\ i=1,2$ are invariant measures of $f_{i}.$
In this case jump  points of $F_{n}$ appear break points of $f_{1}^{q_{n}}.$
 Apply theorem \ref{proposition 4.5} to the function $f_{1}$ we can find sufficiently large $n_{s}=n_{s}(f_{1})>n_{0}$ such that
 the break points of  $f_{1}^{q_{n_{s}}}$  far from each other.
  Let $\delta_{0}=\frac{1}{3}|\log \sigma_{f_{2}}(a_{2})-\log \sigma_{f_{1}}(a_{1})|$  and  $l_{0}=l_{0}(f_{1}, f_{2})$  be the natural number  which is defined in proposition \ref{proposition 4.6}.
   By proposition \ref{proposition 4.6} on one of the two connected components of $\Delta_{f_{1}}^{n_{s}+l_{0}}(a^{1}_{k})\setminus \{a^{1}_{k}\},$ we have
$|\log Df^{q_{n_{s}}}_{2}(h(x))- \log Df^{q_{n_{s}}}_{1}(x)|\geq \delta_{0}.$ By theorem \ref{proposition 4.5}  (ii) there exists $C_{6}=C_{6}(f_{1})$
 such that the length of this component greater then $C_{6}|\Delta_{f_{1}}^{n_{s}+l_{0}}(a^{1}_{k})|.$ Hence, for a suitable subsequence $n_{s}$:\\

$
\ell \big (x: \, |\log Df^{q_{n_{s}}}_{2}(h(x))- \log Df^{q_{n_{s}}}_{1}(x)|\geq \delta_{0}\big)\geq
$
 \begin{eqnarray}\label{eq30}
 \geq \ell \big(x\in \bigcup_{k=0}^{q_{n_{s}}-1}\Delta_{f_{1}}^{n_{s}+l_{0}}(a^{1}_{k}): |\log Df^{q_{n_{s}}}_{2}(h(x))- \log Df^{q_{n_{s}}}_{1}(x)|\geq
\delta_{0}\big)\geq
 \end{eqnarray}
$$
\geq C_{6}\ell \big(\bigcup_{k=0}^{q_{n_{s}}-1}\Delta_{f_{1}}^{n_{s}+l_{0}}(a^{1}_{k})\big)=
C_{6} \sum_{k=0}^{q_{n_{s}}-1}|\Delta_{f_{1}}^{n_{s}+l_{0}}(a^{1}_{k})|.
$$
Using remark \ref{inter nisbat va expon rem}  we get
\begin{eqnarray}\label{eq31}
\sum_{k=0}^{q_{n_{s}}-1}|\Delta_{f_{1}}^{n_{s}+l_{0}}(a^{1}_{k})|
=\sum_{k=0}^{q_{n_{s}}-1}\frac{|\Delta_{f_{1}}^{n_{s}+l_{0}}(a_{k})|}
{|\Delta_{f_{1}}^{n_{s}-1}(a_{k})|}|\Delta_{f_{1}}^{n_{s}-1}(a_{k})|\geq
 \end{eqnarray}
\begin{flushright}
$
 \geq e^{-3v_{1}}\kappa^{l_{0}+1}_{1}\underset{k=0}{\overset{q_{n_{s}}-1}{\sum}}|\Delta_{f_{1}}^{n_{s}-1}(a_{k})|.
$
\end{flushright}
By property \ref{proper1} two consecutive atoms of $\textbf{P}_{n_s}(a_{1}, f_{1})$ are $C_{2}=C_{2}(f_{1})$ -comparable
and using this fact we get
\begin{eqnarray}\label{eq32}
\sum_{k=0}^{q_{n_{s}}-1}|\Delta_{f_{1}}^{n_{s}-1}(a_{k})|=
\frac{\underset{k=0}{\overset{q_{n_{s}}-1}{\sum}}|\Delta_{k}^{n_{s}-1}(a_{1})|}
{\underset{k=0}{\overset{q_{n_{s}}-1}{\sum}}|\Delta_{k}^{n_{s}-1}(a_{1})|+
\underset{k=0}{\overset{q_{n_{s}}-1}{\sum}}|\Delta_{k}^{n_{s}}(a_{1})|}\geq\frac{C_{2}}{1+C_{2}}.
\end{eqnarray}
Using (\ref{eq30}), (\ref{eq31}) and (\ref{eq32}) we get:
\begin{eqnarray}\label{eq33}
  \ell \big (x: |\log Df^{q_{n_{s}}}_{2}(h(x))- \log Df^{q_{n_{s}}}_{1}(x)|\geq \delta_{0}\big)\geq \frac{\kappa_{1}^{l_{0}+1}C_{6}C_{2}}{e^{3v_{1}}(1+C_{2})}.
\end{eqnarray}
Now we consider the second case $\mu_{1}[a_{1}, b_{1}]\neq\mu_{2}[a_{2}, b_{2}].$ Suppose $\mu_{1}[a_{1}, b_{1}]<\mu_{2}[a_{2}, b_{2}]$ the opposite case can be handled similarly. Let $\delta_{1}$ - be the positive number which is defined in proposition \ref{proposition 4.7}.
Using by proposition \ref{proposition 4.7} for any $0\leq k<q_{n_{s}}$ there exits a subinterval $I^{n_{s}}_{k}\subset \Delta_{f_{1}}^{n_{s}+l_{0}}(a^{1}_{k})$
  such that, on the interval $I^{n_{s}}_{k}$ hold the following inequality
$$
|\log Df^{q_{n_{s}}}_{2}(h(x))- \log Df^{q_{n_{s}}}_{1}(x)|\geq
\delta_{1}
$$
 and the intervals $I^{n_{s}}_{k}$ and $\Delta_{f_{1}}^{n_{s}+l_{0}}(a^{1}_{k})$ are $C_{9}$ - comparable. Using similar arguments as above we get:\\

$
\ell \big (x: |\log Df^{q_{n_{s}}}_{2}(h(x))- \log Df^{q_{n_{s}}}_{1}(x)|\big)\geq $
\begin{eqnarray}\label{eq34}
\geq \ell \big(x\in \bigcup_{k=0}^{q_{n_{s}}-1}I_{k}^{n_{s}}: |\log Df^{q_{n_{s}}}_{2}(h(x))- \log Df^{q_{n_{s}}}_{1}(x)|\geq
\delta_{1}\big)\geq \frac{\kappa_{1}^{l_{0}+1}C_{9}C_{2}}{e^{3v_{1}}(1+C_{2})}.
\end{eqnarray}
If we take $\delta=\delta_{1}$ and
$$
\epsilon= \min \{\frac{\kappa_{1}^{l_{0}+1}C_{6}C_{2}}{2e^{3v_{1}}(1+C_{2})}, \,\ \frac{\kappa_{1}^{l_{0}+1}C_{9}C_{2}}{2e^{3v_{1}}(1+C_{2})}\},
$$
 then it is  a contradiction to (\ref{eq29}).
\end{proof}
\section{Acknowledgements}
The first author is  grateful to Universiti Kebangsaan Malaysia for providing financial support via the grants UKM-MI-OUP-2011(13-00-09-001).
The second  author A. Dzhalilov visited to the Universiti Kebangsaan Malaysia  by grant UKM-DIP-2012-31
and the Stonebrook University by Fulbright grant.
He is very grateful to the Universiti Kebangsaan Malaysia and the Stonebrook University for the warm hospitality
and  financial supports. The authors also thank K. Khanin, M. Luybich and M. Martens for useful discussions.

\end{document}